\documentclass[12pt]{amsart}
\usepackage{geometry}         
\usepackage{rotating} 

\usepackage{geometry}                % See geometry.pdf to learn the layout options. %There are lots.

\usepackage{caption}
\usepackage{subcaption}
\usepackage{xcolor}
\usepackage{mathrsfs,amsthm}
\usepackage{amsmath,amssymb}

\usepackage{pinlabel}
\usepackage{fancyhdr}
\usepackage{lastpage}
\usepackage{makecell}
\usepackage{setspace}
\usepackage{longtable}

\usepackage[utf8]{inputenc}

\usepackage{pinlabel,url}
\usepackage{psfrag}
 \usepackage{epsfig}
\usepackage[all]{xy}
\usepackage{tikz}
\usepackage{graphicx}
\usepackage{amsmath,amsthm}
\usepackage{amssymb,amsfonts,latexsym}
\usepackage{epstopdf}
\usepackage[shortlabels]{enumitem}

\usepackage{hyperref}
\hypersetup{
pdfpagemode=FullScreen,
colorlinks=true,
citecolor = blue}

\newtheorem{theorem}{Theorem}[section]
\newtheorem{lemma}[theorem]{Lemma}

\newtheorem{proposition}[theorem]{Proposition}

\theoremstyle{definition}
\newtheorem{definition}[theorem]{Definition}

\newtheorem{conjecture}[theorem]{Conjecture}
\newtheorem{remark}[theorem]{Remark}

\theoremstyle{definition}

\newcommand{\thistheoremname}{}
\newtheorem{genericthm}[theorem]{\thistheoremname}

\newtheorem*{genericthm*}{\thistheoremname}
\newenvironment{namedthm*}[1]
  {\renewcommand{\thistheoremname}{#1}%
   \begin{genericthm*}}
  {\end{genericthm*}}

\def\Vol{\mbox{\rm{Vol}}}

\newcommand\tsim{\kern-.4em\sim}

\newcommand{\Z}{\mathbb Z}

\newcommand{\HH}{\mathbb H}
\newcommand{\Mb}{M_{book}}

\newcommand{\sL}{{\mathcal L}}

\theoremstyle{empty}

\def\spl{\backslash\backslash}

\begin{document}

% Declarations for Front Matter

\title[Guts and The Minimal Volume Hyperbolic 3-Manifold with 3 Cusps]{Guts and The Minimal Volume Orientable Hyperbolic 3-Manifold with 3 Cusps}

 \author[Yue Zhang]{%
        Yue Zhang} 
\address{%
    University of California, Berkeley
} 
\email{% 
     yue\_zhang@berkeley.edu}

\maketitle
% Delete (or comment out) the \approvalpage line for the final version.
%\approvalpage

\begin{abstract}
The minimal volume of orientable hyperbolic manifolds with a given number of cusps has been found for $0,1,2,4$ cusps, while the minimal volume of 3-cusped orientable hyperbolic manifolds remains unknown. By using guts in sutured manifolds and pared manifolds, we are able to show that for an orientable hyperbolic 3-manifold with 3 cusps such that every second homology class is libroid, its volume is at least $5.49\ldots = 6 \times $Catalan’s constant.

\end{abstract}

\section{Introduction}
Jorgensen and Thurston \cite[Section 6.6]{thurston1979geometry} proved that the volumes of hyperbolic 3-manifolds are well-ordered. Moreover, if a volume is an $n$-fold limit point of smaller volumes (of order type $\omega^n$), then there is a corresponding hyperbolic manifold of finite volume with precisely $n$ orientable cusps. This gives rise to the problem of determining the minimal volume orientable hyperbolic 3-manifolds with n cusps.

When $n=0$ (closed manifolds), Gabai, Meyerhoff and Milley  \cite{gabai2009minimum} identified the smallest volume orientable hyperbolic 3-manifold to be the Fomenko-Matveev-Weeks manifold with volume $0.94\ldots$.

When $n=1$, Cao and Meyerhoff \cite{cao2001orientable} showed that the figure-eight knot complement and the manifold obtained by the $(5,1)$-Dehn filling from the Whitehead link complement have the minimal volume. Their volume is $2.02... = 2 V_3$, where $V_3$ is the volume of the ideal regular tetrahedron.

When $n=2$, Agol \cite{agol2010minimal} found that the Whitehead link complement and the $(- 2, 3, 8)$ pretzel link complement are the minimal volume orientable hyperbolic 3-manifolds with two cusps, with volume $V_8 = 3.66\ldots = 4 \times K$, where $K$ is Catalan's constant
 $$K= 1-1/9 + 1/25 -1/49 + \cdots + (-1)^n/ (2n+1)^2+ \cdots.$$ 
$V_8$ is also the volume of a regular ideal octahedron in $\HH^3$.

When $n=4$, Yoshida \cite{yoshida2013minimal} showed that the $8^4_2$ link complement is the unique minimal volume orientable hyperbolic manifold with $4$ cusps.

In the case of 2 and 4 cusps, Agol and Yoshida used guts of pared manifolds to estimate the volume.

However the minimal volume orientable hyperbolic 3-manifolds with 3 cusps remain unknown. Here is a conjecture about the minimal volume.

\begin{conjecture}
	The minimal volume of 3-cusped orientable hyperbolic manifold is the volume of the 3-chain link complement, which is $5.33\ldots$.
	
\end{conjecture}

In this paper, we make some progress on this conjecture.

Let $M$ be an orientable hyperbolic 3-manifold with cusps and $z$ be an element in $H_2(M,\partial M)$.

Proposition \ref{prop:gutcomponentinhyperbolic} classifies sutured guts component of $\Gamma(z)$ defined in \cite{AZ1}. Each guts component is a $T\times I$ with a boundary component as a toral suture and two annuli on the other boundary component as annular sutures, a solid torus with 2 non-longitudinal sutures in the boundary, a solid torus with 4 longitudinal sutures in the boundary, or an acylindrical taut sutured manifold.

 We say $z$ is libroid if the sutured guts $\Gamma(z)$  consists of solid tori and $T\times I$'s. See Definition \ref{def:libroid}. If $z$ is libroid, there is a \emph{library bundle structure} of $M$ such that the spine is $\Gamma(z)$. By \emph{exchanging} guts components of $\Gamma(z)$, we are able to have two guts in the adjacent \emph{layers} \emph{intersect} each other. This provides some complicated pared guts in the manifold. See Section \ref{sec:Library Bundles} for definitions and Lemma \ref{lem:intersects} for a more precise statement.
 
 We also find the pared guts of library sutured manifolds with $2,3$ or 4 layers in Lemma \ref{lem:gutsoflibrary}.

By applying Lemma \ref{lem:intersects} and Lemma \ref{lem:gutsoflibrary} as well as techniques in \cite{agol2010minimal} and \cite{yoshida2013minimal}, we prove the following theorem in Section \ref{sec:volume:volumeoflibrary}.

\begin{theorem} \label{thm:2TIor4ST}
	Let $M$ be an orientable hyperbolic 3-manifold of finite volume with 3 cusps such that each class in $H_2(M,\partial M)$ is libroid. Let $z$ be a vertex of the Thurston sphere. If $\Gamma(z)$ contains two $T\times I$'s or 4-STs $G$ and $G'$ such that each suture of $G$ and $G'$ is homologically nontrivial and each suture of $G$ is not isotopic to any suture of $G'$, then $\Vol(M) \ge 1.5 V_8.$ 
\end{theorem}

We define doubled guts for $z$ as the \emph{horizontally prime guts} of the double of $\Gamma(z)$ in Definition \ref{def:doubledguts}. Let $\Delta$ be an open Thurston cone. We prove in Theorem \ref{thm:invariancedoubledguts} that if any element $z$ in $H_2(M,\partial M)$ is a libroid class and no element in $\Delta$ \emph{vanishes} (see Definition \ref{def:vanish}) on two boundary component, the doubled guts of elements in $\Delta$ remain invariant. This is a stronger result than \cite[Theorem 1.2]{AZ1} for sutured guts. 

By using doubled guts, we show that if $M$ is a hyperbolic 3-manifold with 3 cusps such that every element in $H_2(M,\partial M)$ is a libroid class and $Vol(M) < 2 V_8$, there is a vertex $z$ of the Thurston sphere that satisfies the condition of Theorem \ref{thm:2TIor4ST}. Hence we prove Theorem \ref{thm:volumeofbook}. 

\begin{theorem} \label{thm:volumeofbook}
	Let $M$ be an orientable hyperbolic 3-manifold of finite volume with 3 cusps such that every element in $H_2(M,\partial M)$ is a libroid class. Then $\Vol(M) \ge 1.5V_8.$
	\end{theorem}

Note that the volume of the 3-chain link $C_3$ complement $= 5.33\ldots <5.49\ldots = 1.5 V_8$. %\cite[Example 6.8.1]{thurston1979geometry}. 
So if an orientable hyperbolic 3-manifold $M$ with 3 cusps has volume smaller than $5.33\ldots$, there is an element $z$ in $H_2(M,\partial M)$ such that $\Gamma(z)$ contains an acylindrical taut sutured manifold.

If we cut the chain link complement along a 3-punctured sphere, we will get the acylindrical sutured manifold $(N_3,\gamma)$ in figure \ref{fig:acylindrical}. We can think of it as a pared manifold with the pared locus as the sutures.

\begin{figure}[hbt]
\begin{center}
	  \includegraphics[width=2 in]{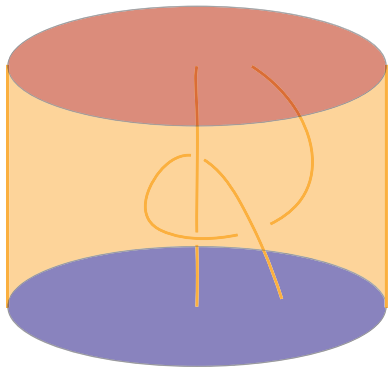}
\end{center}
\caption{Acylindrical sutured manifold $(N_3,\gamma)$.}\label{fig:acylindrical}
\end{figure}

This acylindrical sutured manifold has volume $5.33\ldots$ (see the figure in \cite[Example 6.8.1]{thurston1979geometry}). We state a conjecture for it.

\begin{conjecture}\label{con:acylindricalsutured}
	$(N_3,\gamma)$ is a minimal volume acylindrical taut sutured manifold up to orientation.
\end{conjecture}

By Theorem \ref{thm:AST}, we actually show that 
\begin{theorem}
Suppose Conjecture \ref{con:acylindricalsutured} is true. Then the smallest volume of 3-cusped orientable hyperbolic manifolds is the volume of the 3-chain link. 
\end{theorem}

This paper is based on and expands upon Chapter 2 of the author's Ph.D. thesis \cite{zhang2020guts}.

{\bf Organization.}

Section \ref{sec:volume:pared} consists of a review of pared manifolds and pared guts.

In Section \ref{sec:volume:sutured}, we classify sutured guts in manifolds whose interior is a hyperbolic manifold of finite volume. Moreover, we show some properties of the sutured guts related to a vertex element of the Thurston norm.

In Section \ref{sec:volume:library}, we introduce the definition of sutured bundles and library bundles which is used to prove that we cannot exchange sutured guts components infinitely many times in a hyperbolic manifold (Lemma \ref{lem:intersects}). We also define pared sutured manifolds which have compatible sutured manifold structure and pared manifold structure which helps find the guts of library sutured manifolds.

In Section \ref{sec:volume:3-cusped}, we show that for a 3-cusped hyperbolic manifold of finite volume, there is a sequence of 2 sutured decomposition for the sutured guts related to different 2-homology classes. 

In Section \ref{sec:volume:volumeoflibrary}, we analyze different cases for a 3-cusped hyperbolic manifold under a non-isotopic condition and estimate the volume of it as in Theorem \ref{thm:2TIor4ST}.

In Section \ref{sec:volume:general}, we define doubled guts for 2-homology classes. In Theorem \ref{thm:invariancedoubledguts}, we show that the doubled guts remain invariant in any open Thurston cone $\Delta$ for a certain class of orientable hyperbolic 3-manifolds. By using doubled guts, we deduce Theorem \ref{thm:volumeofbook} to the case in Section \ref{sec:volume:volumeoflibrary}.

{\bf Acknowledgements.}
	
I am grateful to my Ph.D. advisor, Ian Agol, for his guidance, support, and mentorship throughout the development of this paper and my thesis. He introduced me to the problems of minimal volume hyperbolic manifolds and suggested the approach of using guts in 3-dimensional manifolds. Furthermore, he provided valuable suggestions and identified issues in previous drafts, which were crucial to the completion of this work. I would also like to thank Ken’ichi Yoshida, Jason DeBlois, Sadayoshi Kojima and David Gabai for helpful comments and conversations. 

\section{Pared Manifolds, Pared Guts and Volume Estimation} \label{sec:volume:pared}
We call the guts in \cite{agol2010minimal} \emph{pared guts} since the notion comes from the characteristic manifold of a pared manifold.

\begin{definition}
A {\it pared manifold} is a  pair $(M,P)$ where 
\begin{itemize}
\item
$M$ is a compact, orientable irreducible 3-manifold and
\item
$P\subset \partial M$ is a union of essential annuli
and tori in $M$,
\end{itemize}
such that
\begin{itemize}
\item
every abelian, noncyclic subgroup of $\pi_1(M)$ is peripheral
with respect to $P$ (i.e., conjugate to a subgroup of the fundamental
group of a component of $P$) and
\item 
every map $ \varphi:(S^1\times I, S^1\times \partial I) \to (M,P)$ that
is injective on the fundamental groups deforms, as maps of pairs, into $P$. 
\end{itemize}
$P$ is called the {\it parabolic locus} of the {pared manifold} $(M,P)$.
We denote by $\partial_0 M$ the surface $\partial M - int(P)$. 
\end{definition}

\begin{lemma}[{\cite[Lemma 3.2]{agol2010minimal}}]\label{lem:pared}
	Let $M$ be an orientable compact manifold such that $int(M)$ is hyperbolic of finite volume and so that $P = \partial M$ is the pared locus of $M$. Let $X \subset M$ be an essential surface. Then $(M\spl X,P\spl \partial X)$ is a pared manifold.
\end{lemma}

Let $(M,P)$ be a {pared manifold} such that $\partial_0 M$ is incompressible.
There is a canonical set of essential annuli $(A,\partial A)\subset (M,\partial_0 M)$, called the {\it characteristic annuli}, 
such that $(P,\partial P)\subset (A,\partial A)$, and
characterized (up to isotopy) by the property that they are the maximal collection of non-parallel essential
annuli such that every other essential annulus $(B,\partial B)\subset (M,\partial_0 M)$
may be relatively isotoped to an annulus $(B',\partial B')\subset (M,\partial_0 M)$
so that $B'\cap A=\emptyset$.
Each complementary component  $L\subset M\spl A$  is 
one of the
following types:

\begin{enumerate}
\item
$T^2\times I$, a neighborhood of a torus component of $P$,
\item 
$(S^1\times D^2, S^1\times D^2 \cap \partial_0 M)$, a solid torus with annuli in the boundary,
\item
$(I-bundles, \partial I -subbundles)$, where the $I$-bundles over
the boundary are subsets of $A$, or
\item
all essential annuli in $(L,\partial_0 M\cap L)$ are parallel in $L$ into 
$(L\cap A, \partial(L\cap A) ).$
\end{enumerate}

The \hypertarget{window}{union} of components
of type $(3)$, denoted $(W,\partial_0 W)\subset (M, \partial_0 M)$, is called the {\it window} of $(M,\partial_0 M)$. 
It is unique up to isotopy of pairs. 

The union of the components of type 4 is called the \emph{pared guts} of $M$ and denoted by $Guts(M,P)$ (a bit different from \cite{agol2010minimal}). Note that the parabolic locus of $Guts(M,P)$ will consist of characteristic annuli. If $M$ is compact {orientable} and $int(M)$ admits a metric of finite volume, and $(X,\partial X)\subset (M,\partial M)$ is an essential surface,
then define $Guts(M,X)=Guts(M\spl X,\partial M\spl \partial X)$. 
The components of type (4) are acylindrical {pared manifolds}, which have a complete hyperbolic
structure of finite volume with geodesic boundary \cite{morgan1984chapter}. 
We let $\Vol(Guts(M,P))$ denote the volume of this hyperbolic metric. If $D(M,P)$ is
obtained by taking two copies of $M$ and gluing them along the corresponding surfaces
$\partial_0 M$, then $\Vol(Guts(M,P))=\frac{1}{2} \Vol( D(M,P) )$, where $\Vol(D(M,P))$ is
the simplicial volume of $D(M,P)$, i.e., the sum of the volumes of the hyperbolic
pieces of the geometric decomposition.

In this paper, we heavily use the following theorem to estimate the volume of a hyperbolic manifold of finite volume.

\begin{theorem}[{\cite[Theorem 5.5]{calegari2010positivity}, \cite[Theorem 9.1]{agol2005lower}}]
\label{thm:AST}
 
Let $M$ be a finite volume orientable hyperbolic manifold, and $S \subset M$ be 
an essential surface. Then 
\[
\mathrm{Vol}(M) \geq \mathrm{Vol}(Guts(M,S)) \geq \frac{V_8}{2} 
|\chi (\partial Guts(M,S))|. 
\]

\end{theorem}

The estimation of $\mathrm{Vol}(Guts(M,S))$ in 
Theorem \ref{thm:AST} follows from the following theorem. 

\begin{theorem}[{\cite[Theorem 5.2]{miyamoto1994volumes}}]
\label{thm:Miyamoto}

Let $M$ be a hyperbolic manifold with totally geodesic boundary. Then 
$\mathrm{Vol}(M) \geq \frac{V_8}{2} |\chi (\partial M)|$. Moreover, $M$ is obtained  
from ideal regular octahedra by gluing along their faces when the equality holds. 
\end{theorem}

The following two results come from \cite{yoshida2013minimal} which is about the minimal volume orientable hyperbolic 3-manifold with 4 cusps. We will also use them to estimate the volume of some pared guts.
\begin{lemma}[{\cite[Lemma 3.2]{yoshida2013minimal}}]
\label{lem:nonsepsurface}
Let $L$ be an orientable hyperbolic 3-manifold with geodesic boundary $S$, with $k$ 
annular cusps $A_1 , \dots , A_k$ and with $n-k$ torus cusps $T_{k+1} , \dots , T_n$, 
where $1 \leq k \leq 3$ and $n \geq 4$. Assume that $\chi (S) = -2$. 
Then there is an essential surface $Y \subset L$ such that 
$Y \cap S = \emptyset$ and $[\partial Y] \neq 0 \in H_1 (\partial L; \mathbb{Z})$. 
\end{lemma}

\begin{theorem}[{\cite[Theorem 4.1]{yoshida2013minimal}}]
	\label{thm:estimategeodbd} 
Let $L$ be an orientable hyperbolic 3-manifold with geodesic boundary $S$. 
Suppose that there is an essential surface $Y \subset L$ such that 
$Y \cap S = \emptyset$ and $[\partial Y] \neq 0 \in H_1 (\partial L; \mathbb Z)$. 
Then there is an essential surface $Y^{\prime}$ such that 
$\chi (\partial \mathrm{Guts}(L \spl Y^{\prime})) \leq -4$ and 
$\mathrm{vol}(L) \geq 2V_8$. 
\end{theorem}

\begin{theorem} \label{thm:4cusps}
	Let $L$ be an orientable hyperbolic 3-manifold with geodesic boundary $S$, with $k$ 
annular cusps $A_1 , \dots , A_k$ and with $n-k$ torus cusps $T_{k+1} , \dots , T_n$, where $1 \leq k \leq 3$ and $n \geq 4$. Then $\mathrm{vol}(L) \geq 2V_8$. 
\end{theorem}
\begin{proof}
Since the double of $L$ along $S$ has Euler characteristic 0, $\chi(S) = 2 \chi(L)$. 
If $\chi(S) \le -4$, by Theorem \ref{thm:AST}, $\Vol(L) \ge 2V_8$. If $\chi(S) = -2$, by Lemma \ref{lem:nonsepsurface}, there is an essential surface $Y \subset L$ such that 
$Y \cap S = \emptyset$ and $[\partial Y] \neq 0 \in H_1 (\partial L; \mathbb Z)$. Then by Theorem \ref{thm:estimategeodbd}, $\mathrm{vol}(L) \geq 2V_8$.
\end{proof}

We also state the definition of an annular compression from \cite{agol2010minimal} which will help the estimation of volume.

Let $(X, \partial X) \subset (M,\partial M)$ be a properly embedded essential surface. A {\it compressing annulus} for $X$ is an embedding 
$$ i: (S^1\times I, S^1\times\{0\} , S^1\times \{1\}) \hookrightarrow (M, X, \partial M)$$
such that 
\begin{itemize}
\item
$i_{\ast}$ is an injection on $\pi_1$,
\item 
$i(S^1\times I) \cap X = i(S^1\times \{0\})$, and 
\item
$i(S^1\times\{0\})$ is not {isotopic} in $X$ to $\partial X$. 
\end{itemize}

An {\it annular compression} of $(X,\partial X) \subset (M,\partial M)$ is a surgery 
along a compressing annulus $ i: (S^1\times I, S^1\times\{0\} , S^1\times \{1\}) \hookrightarrow (M, X, \partial M)$ for $X$.

 Let $U$ be a regular neighborhood of $i(S^1\times I)$ in $M\spl X$, let $\partial_1 U$ be the frontier of $U$ in
$M\spl X$, and let $\partial_0 U=\partial U \cap (X\cup \partial M)$. Then let $X'= (X-\partial_0 U)\cup \partial_1 U$. 
The surface $X'$ is the \emph{annular compression} of $X$. We remark that if $X$ is essential, then $X'$ is as well.

\section{Sutured Manifolds and Sutured Guts}\label{sec:volume:sutured}

We call the guts defined in \cite{AZ1} \emph{sutured guts} since they are related to sutured manifolds. In this section, we set up some notation and terminology for the theory of sutured decompositions and sutured guts.

We classify the sutured guts components of a hyperbolic manifold of finite volume. 
\begin{definition}
	We call a sutured manifold $(N,R_+,R_-,\gamma)$ \emph{acylindrical} if $N$ is atoroidal and every essential annulus in $(N,R_+\cup R_-)$ is parallel in $N$ into $(\gamma,\partial \gamma)$ rel $R_+\cup R_-)$.
\end{definition}

\begin{proposition}[{Cf. \cite{agol2010minimal}}]\label{prop:gutcomponentinhyperbolic}  Let $M$ be a compact 3-manifold with toral boundary such that its interior is a hyperbolic manifold of finite volume. Each guts component with respect to a facet surface is one of the following types:
	\begin{enumerate}

\item a $T\times I$ with a boundary component as a sutured torus and two annuli on the other boundary component as sutures,
%$(T^2\times I,T^2\times \{0\} \cup \gamma)$, 

\item  a solid torus with 2 non-longitudinal sutures in the boundary,
%$(S^1\times D^2, \gamma)$,
\item 
a solid torus with 4 longitudinal sutures in the boundary, or
%$(S^1\times D^2, \gamma)$, 
\item an acylindrical taut sutured manifold.
\end{enumerate}

\end{proposition}

 \begin{figure}[hbt]
 \begin{center}
 		\includegraphics[height = 1 in]{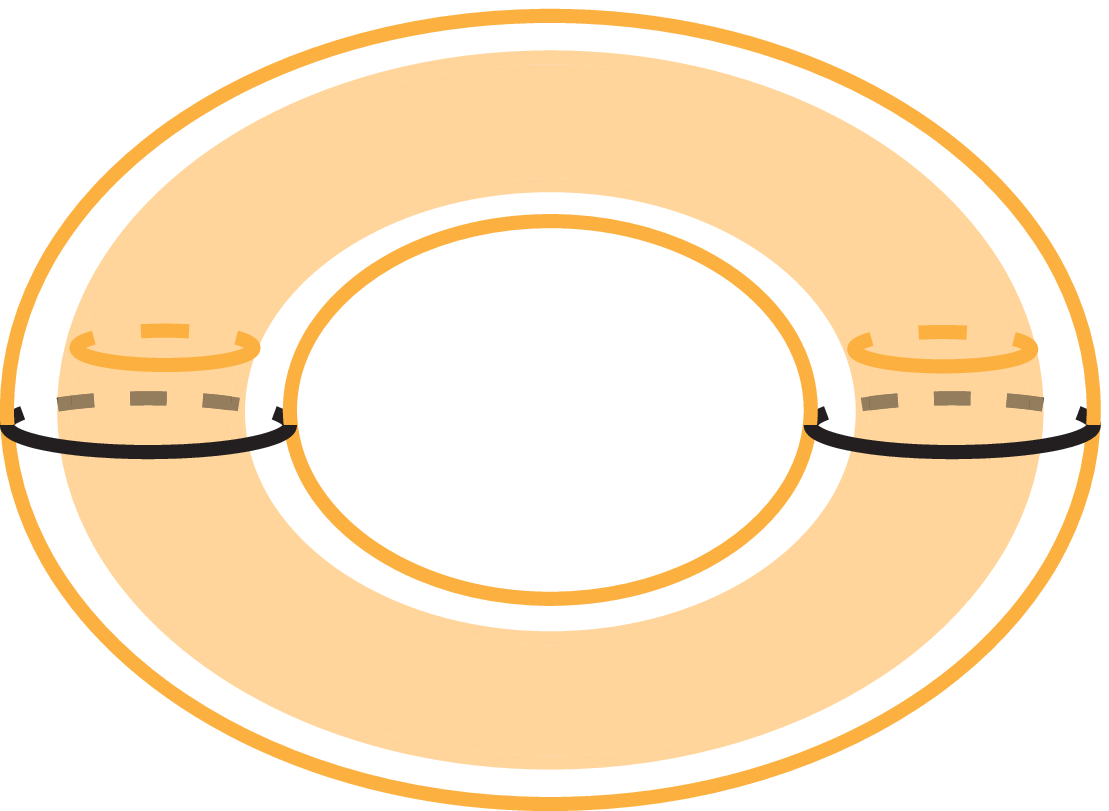}\quad
	\includegraphics[height = 1 in]{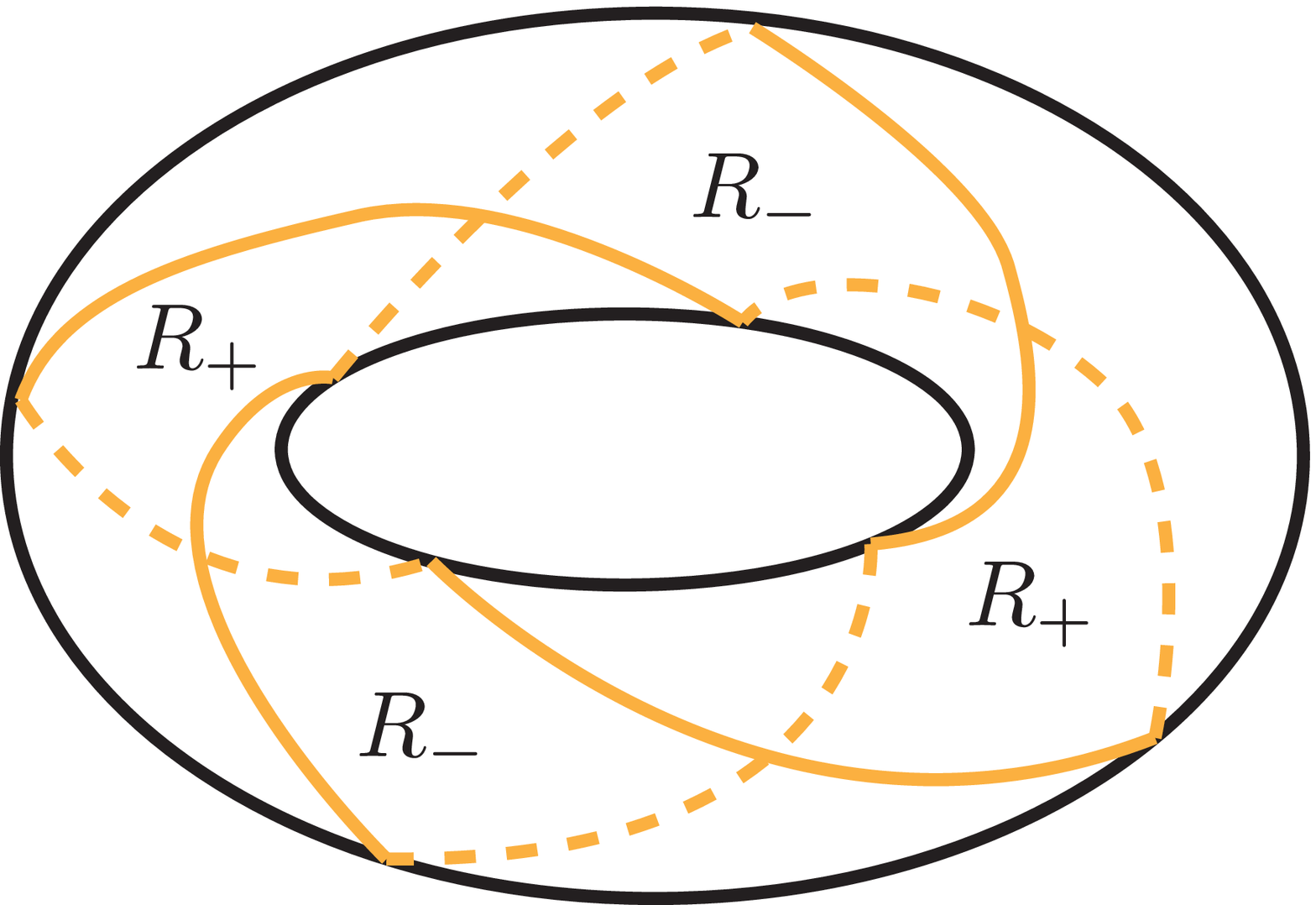}\quad
	\includegraphics[height = 1 in]{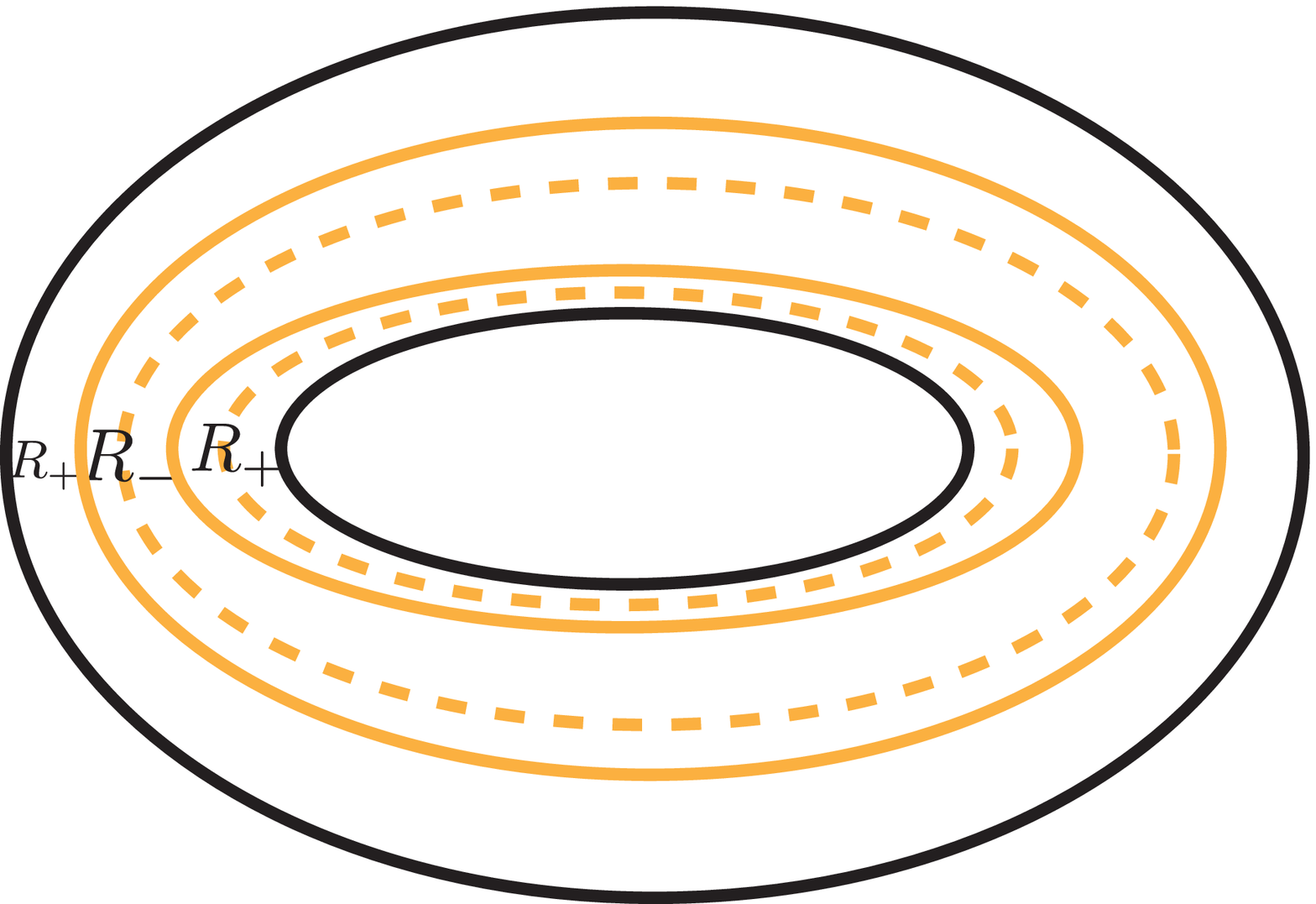}
	\caption{$T\times I$, 2-ST and 4-ST.}
 \end{center}

\end{figure}

\begin{proof}

	Let $(G,R_+,R_-,\gamma)$ be a guts component with respect to a facet surface $F$. Because of \cite[Lemma 3.6]{AZ1}, $G$ is reduced and horizontally prime. Hence it does not have essential product annulus.

	Suppose there is a non-product annulus $A$. Without loss of generality, we assume its boundary lies on $R_+$ of $G$. Then we do a cut-and-paste surgery for $R_+$ along $A$ to have $R'$. Because $R'$ is homologous to $R_+$ in $H_2(G,\gamma;\Z)$ and $\chi_-(R')= \chi_-(R_+)$, $R'$ is a horizontal surface in $G$. Because $G$ is horizontally prime, $R'$ is parallel to $R_-$. Then $\partial A$ induce product annuli that connect $R_+$ and $R_-$ and these annuli are parallel to the sutures of $G$. Therefore $R_+$ consists of 2 copies of annuli and $G$ is a solid torus with 4 longitudinal sutures. 
	 
	 In the following, we consider the case when $G$ is acylindrical. Hence by the statement about characteristic annuli in \cite[Section 2]{agol2010minimal}, $G$ is a $T^2 \times I$ which is a neighborhood of a component of $\partial M$, a solid torus with annuli in the boundary or an acylindrical pared manifold.
	 
	 If $G$ is a $T^2 \times I$, $G$ cannot have more than 2 annular sutures on a boundary component. Otherwise, there is a horizontal surface which is an annulus and not parallel to boundary. See Figure \ref{fig:TI}. Hence $G$ has a boundary component as a toral suture and two annuli on the other boundary component as annular sutures.
	 \begin{figure}[hbt]
	\begin{center}
  \includegraphics[width = 2 in]{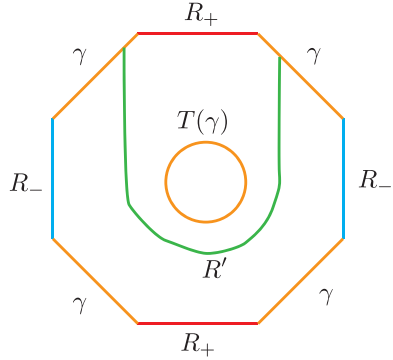}
  \caption{A cross section of $T\times I$. $R'$ is an annulus and not parallel to boundary} \label{fig:TI}
  \end{center}
\end{figure}

	 Let $G$ be a solid torus. If $G$ have more than 2 sutures, there is a non-trivial horizontal surface which is an annulus as in the case of $T^2 \times I$. Since $G$ is not a product sutured manifold, the sutures on $\partial G$ are not the longitudes of $G$. Hence $G$ is a solid torus with 2 non-longitudinal sutures in the boundary.
	 \end{proof}

In the following we sometimes use the core of an annulus to represent the annulus.

\begin{definition}
	We define a \emph{2-ST} as a solid torus with 2 non-longitudinal sutures in the boundary and a \emph{4-ST} as a solid torus with 4 longitudinal sutures in the boundary.
\end{definition}

\begin{definition}
	We call an element in $H_2(M,\partial M)$ \emph{a vertex element of the Thurston norm} if it is a multiple of a vertex of the Thurston sphere.
\end{definition}

We are interested in the sutured guts related to a vertex element because it has the most complexity.

\begin{lemma} \label{lem:connected}
	A norm-minimizing surface $S$ presenting a primitive vertex element of the Thurston norm is connected.
\end{lemma}
\begin{proof}
	If $S$ is not connected, $S$ consists of two surfaces $S_1$ and $S_2$. Then $x_M(S_i) \le \chi^+(S_i)$ and hence \[x_M(S) \le x_M(S_1) + x_M(S_2) \le \chi^+(S_1) + \chi^+(S_2) = \chi^+(S)= x_M(S)\] 
	which means $S_1$ and $S_2$ are norm-minimizing surfaces. For any $a \ge b\ge 0$, $a x_M(S_1) + bx_M(S_2) \ge x_M(aS_1 + bS_2) \ge x_M(a(S_1 + S_2)) - x_M((a-b)S_2) = ax_M(S_1) + bx_M(S_2)$ which means $x_M(aS_1 + bS_2) = ax_M(S_1) + bx_M(S_2)$. Similarly the equation works for $a < b$. However, this violates the condition that $[S]$ is a vertex element of the Thurston face.
\end{proof}

By Lemma \ref{lem:connected} and \cite[Lemma 3.7]{AZ1}, we have the following lemma.
\begin{lemma} \label{lem:label} 
	If $S$ is a facet surface for a vertex element of the Thurston norm, each component of $M\spl S$ contains exactly one guts component.

\end{lemma}

\begin{definition}
We use the notation $M \stackrel{S}{\overline \leadsto} N$ to denote that we do a sutured decomposition for $M$ along $S$ and throw away product sutured manifold components to have $N$. If $M$ and $N$ are taut, we call it a \emph{taut sutured decomposition}.

	We say a taut sutured manifold $M_0$ is \emph{depth} $k$ if there is a sequence of $k$ nontrivial taut sutured decompositions 
	\[
	M_0 \stackrel{S_1}{\overline \leadsto} M_1 \leadsto \cdots \stackrel{S_k}{\overline \leadsto} M_k.
	\]

\end{definition}

\section{Sutured Bundles and Library Bundles} \label{sec:Library Bundles} \label{sec:volume:library}

 In this section, we introduce some general settings that will help with analyzing the volume of hyperbolic manifolds.

\begin{definition} \label{def:suturedbundle}
	Let $M$ be an orientable, irreducible, $\partial$-irreducible manifold with toral boundary and $F$ be a properly norm-minimizing surface representing multiple of a primitive element $z$ of the Thurston norm. Then $[F] = kz$ for some positive integer $k$. 
	
	Choose a base point $p$. For a point $q$ in $M\backslash F$, we define the \emph{potential} function $\Phi(q)$ to be the intersection number of a curve connecting $p,q$ and $F$ mod $k$. The potential function $\Phi: M\backslash F \to \Z/k\Z$ is well-defined and we can assign each component of $M\spl F$ to a number in $\Z/k\Z$ (cf. \cite[Lemma 3]{gabai1986genera}). Then for a number $i$ in $\Z/k\Z$, we say the $i$-th \emph{layer} is the union of points in $M\spl F$ having the value of $\Phi$ as $i$. We think of a layer as a sutured manifold with $R_\pm$ as a union of components of $F$ and denote it as $\sL_i(F)$. 
	
	We say the structure defined here is a \emph{sutured bundle} structure for $M$ with respect to $F$. We let $\Sigma_{i}$ be the $R_-$ part of $\sL_i(F)$ and call it a \emph{horizontal surface} of the sutured bundle. Then $F = \cup_{i=0}^{k-1} \Sigma_{i}$.
	
\begin{figure}[hbt]
\includegraphics[height = 1.5 in]{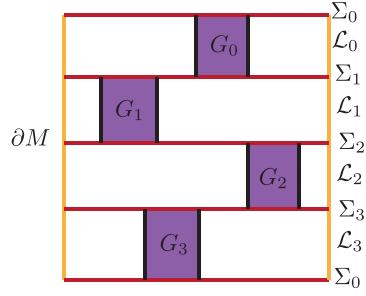}
\caption{Sutured bundle structure and horizontal surfaces}
\end{figure}	
	
\end{definition}

\begin{definition}[Cf. {\cite[Section 6.3]{agol2015certifying}}]\label{def:librarybundle}
	Let $\{\sL_{j}(F) | 0 \le i \le k-1 \}$ be a sutured bundle structure for $M$ with respect to $F$. We denote $G_i$ be the sutured guts component of $\sL_i(F)$. We call the union of $G_i$ as the \emph{spine} of the sutured bundle. We say $\Sigma_i$ is the \emph{horizontal surface} from $G_{i-1}$ to $G_i$. If each $G_i$ is a union of solid tori and $T\times I$'s, we say $F$ give $M$ a structure of a \emph{library bundle}.
\end{definition}

For a layer $\sL_i(F)$ related to a gut $G_i$ as a 4-ST, $\sL_i(F)$ is a union of $G_i$ and a product sutured manifold. We take a union $R'$ of two annulus in $G_i$ such that $R',\gamma$ and $R_-$ bound a product sutured manifold $R_-\times I$. Furthermore, $R'$ and $R_+$ bound a thickened annulus. See Figure \ref{fig:4ST}.

\begin{figure}[hbt]
\begin{center}
\includegraphics[width = 2 in]{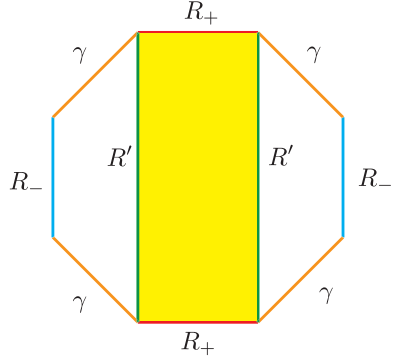}
  \caption{A cross section of a 4-ST.}\label{fig:4ST}
\end{center}
 
\end{figure}

Hence the layer comes from gluing a thickened annulus to $\Sigma_i\times I$ (or $\Sigma_{i+1}\times I$) via $\Sigma_i\times \{1\}$ (or $\Sigma_{i+1}\times \{0\}$). Furthermore, the two boundary curves of the annulus are nonparallel closed curves on $\Sigma_i$ (or $\Sigma_{i+1}$). In fact, the two boundary curves of the annulus are the cores of $R_+$ part of $G_i$ (or the $R_-$ part of $G_{i}$). 

For a layer related to a gut $G_i$ which is a 2-ST or a $T\times I$, it can be constructed as follows. We pick a nontrivial closed curve $c_i$ on $\Sigma_i$ and do a $\beta_i$ Dehn surgery on $\xi_i = c_i\times\{ 1/2 \}$ in $\Sigma_i \times I$ where $\beta_i = q_i/p_i$ with $|q_i|\ge 2$ when $G_i$ is a 2-ST or $\beta_i =\infty$ when $G_i$ is a $T\times I$ (See \cite[Example 7.1]{AZ1}).

\begin{remark}
	Let $\{\sL_{j}(F) | 0 \le i \le k-1 \}$ be a sutured bundle structure for $M$ with respect to $F$. If the guts $G_i$ of $\sL_{i}(F)$ does not intersect the guts $G_{i+1}$ of $\sL_{i+1}(F)$, we can extend a boundary component of each sutured annulus of $G_i$ to $\sL_{i+1}(F)$ and denote the union of the new annuli as $A$. Then by cutting $\sL_i(F(z)) \uplus \sL_{i+1}(F(z))$ along $A$, we will have $G_i$ and its complement up to homeomorphism.
\end{remark}

\begin{definition}
	We think of $R_+(G_i)$ and $R_-(G_{i+1})$ are both subsurfaces of $\Sigma_{i+1}$. If $R_+(G_i)$ is disjoint with $R_-(G_{i+1})$, we can do a cut-and-paste operation for $F$ to have a new $F'$ by replacing $\Sigma_{i+1}$ with $\Sigma_{i+1}\backslash( R_-(G_{i+1}) \cup R_+(G_i))$ glued by $( R_+(G_{i+1}) \cup R_-(G_i))$ and the annular components of $(\gamma(G_{i+1})\cup\gamma(G_{i}))$. We can relabel the guts with respect to $F'$ so that we have $G_0,\ldots,G_{i+1},G_{i},\ldots,G_{k-1}$ in sequence. In the following, we mean \emph{exchanging} $G_i$ and $G_{i+1}$ by this specific operation. We say $G_i$ and $G_{i+1}$ is \emph{exchangeable} if we can exchange $G_i$ and $G_{i+1}$.
	
\begin{figure}[hbt]
\begin{subfigure}[t]{.25\textwidth}
\begin{center}
\includegraphics[width=1.3 in]{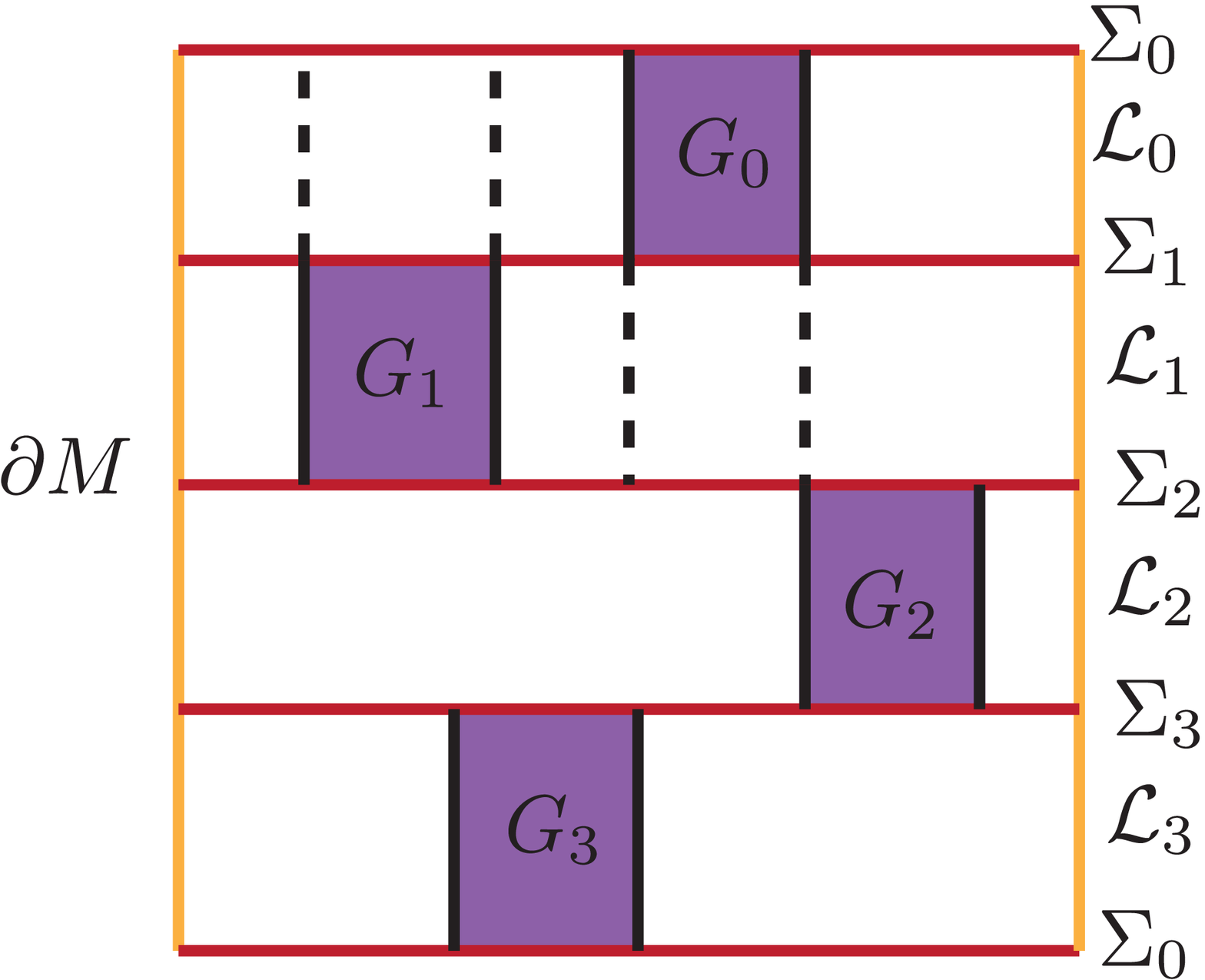}
\caption{}
\end{center}
\end{subfigure}%
\begin{subfigure}[t]{.25\textwidth}
\begin{center}
\includegraphics[width=1.3 in]{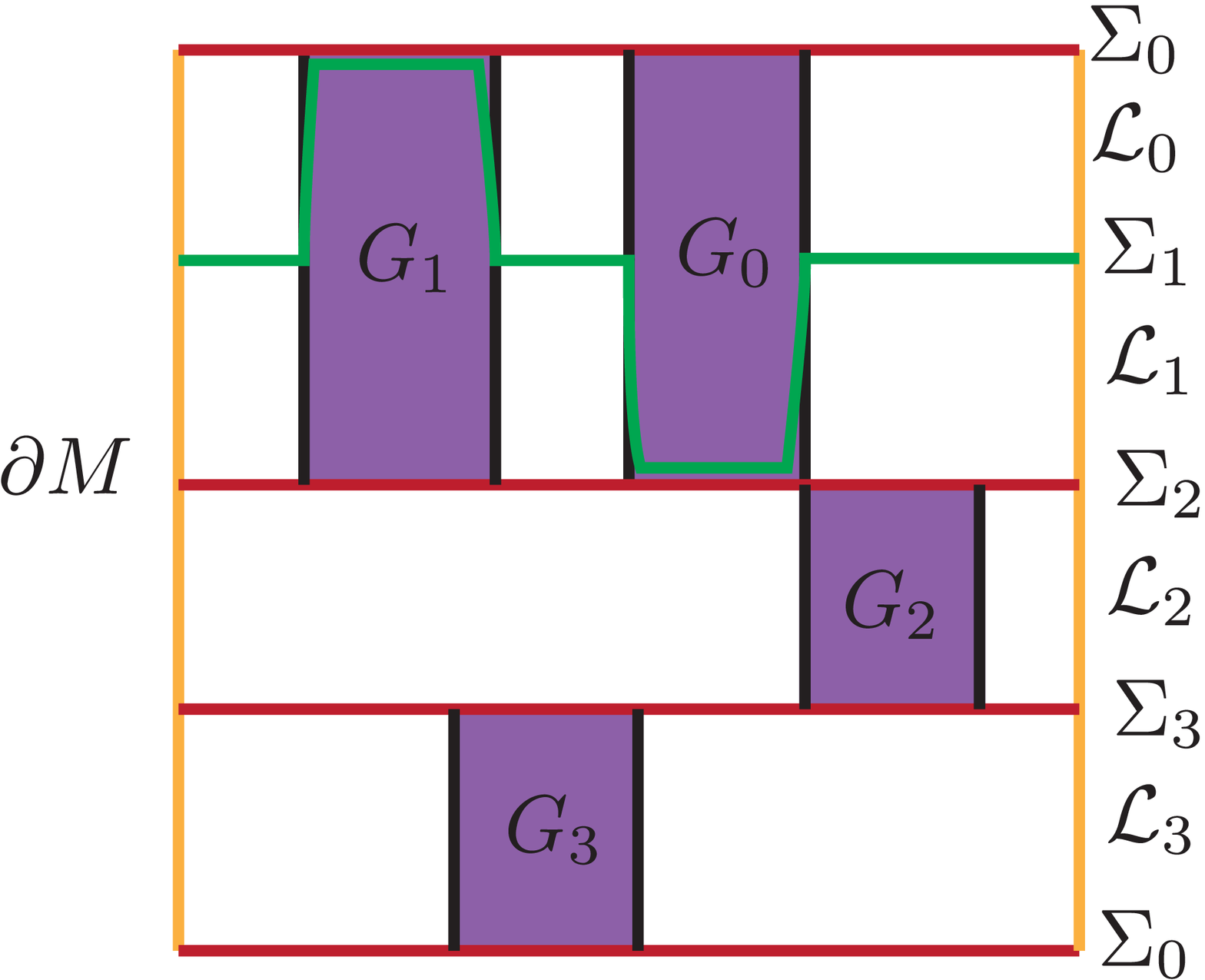}
\caption{}
\end{center}
\end{subfigure}%
\begin{subfigure}[t]{.25\textwidth}
\begin{center}
\includegraphics[width=1.3 in]{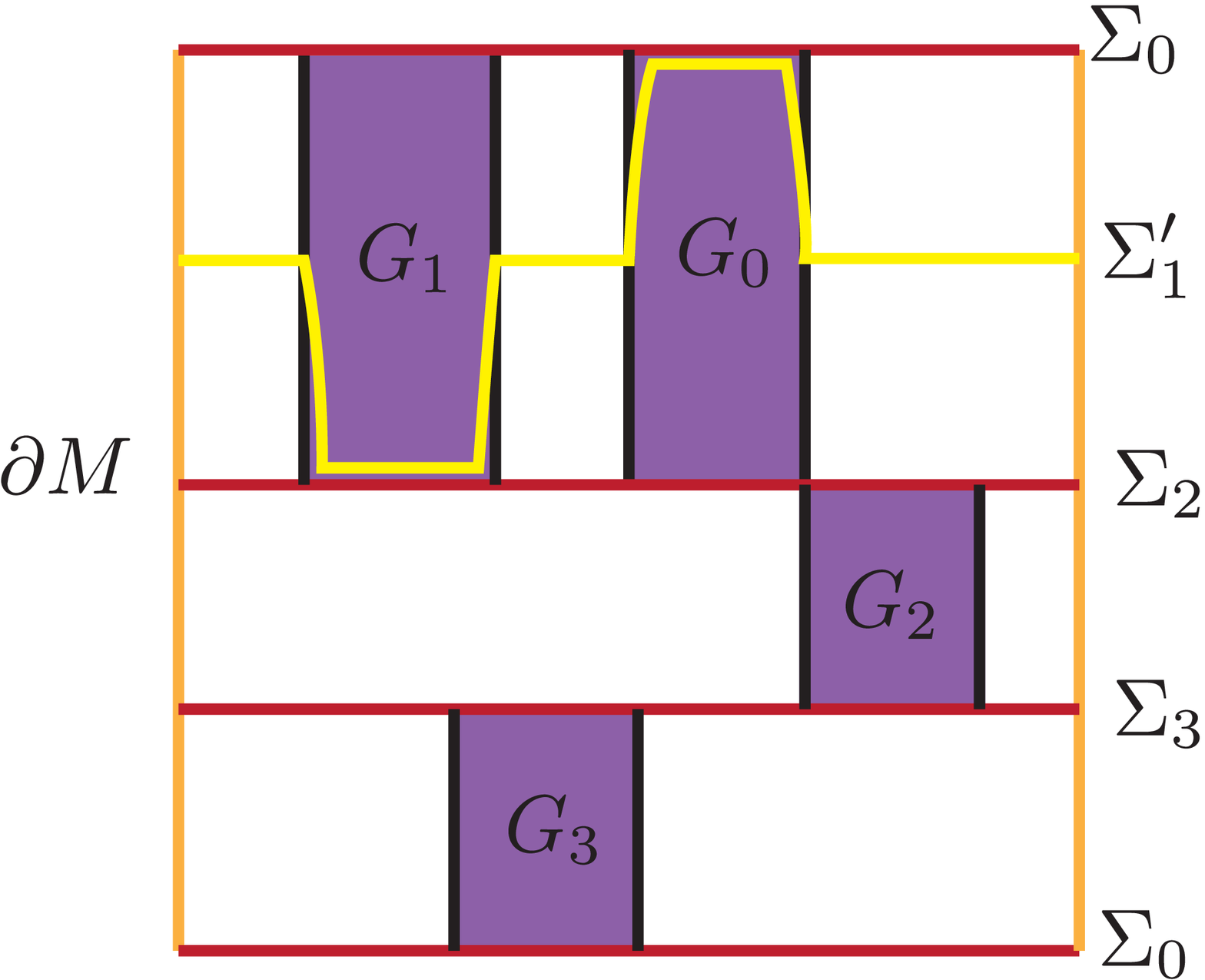}
\caption{}
\end{center}
\end{subfigure}%
\begin{subfigure}[t]{.25\textwidth}
\begin{center}
\includegraphics[width=1.3 in]{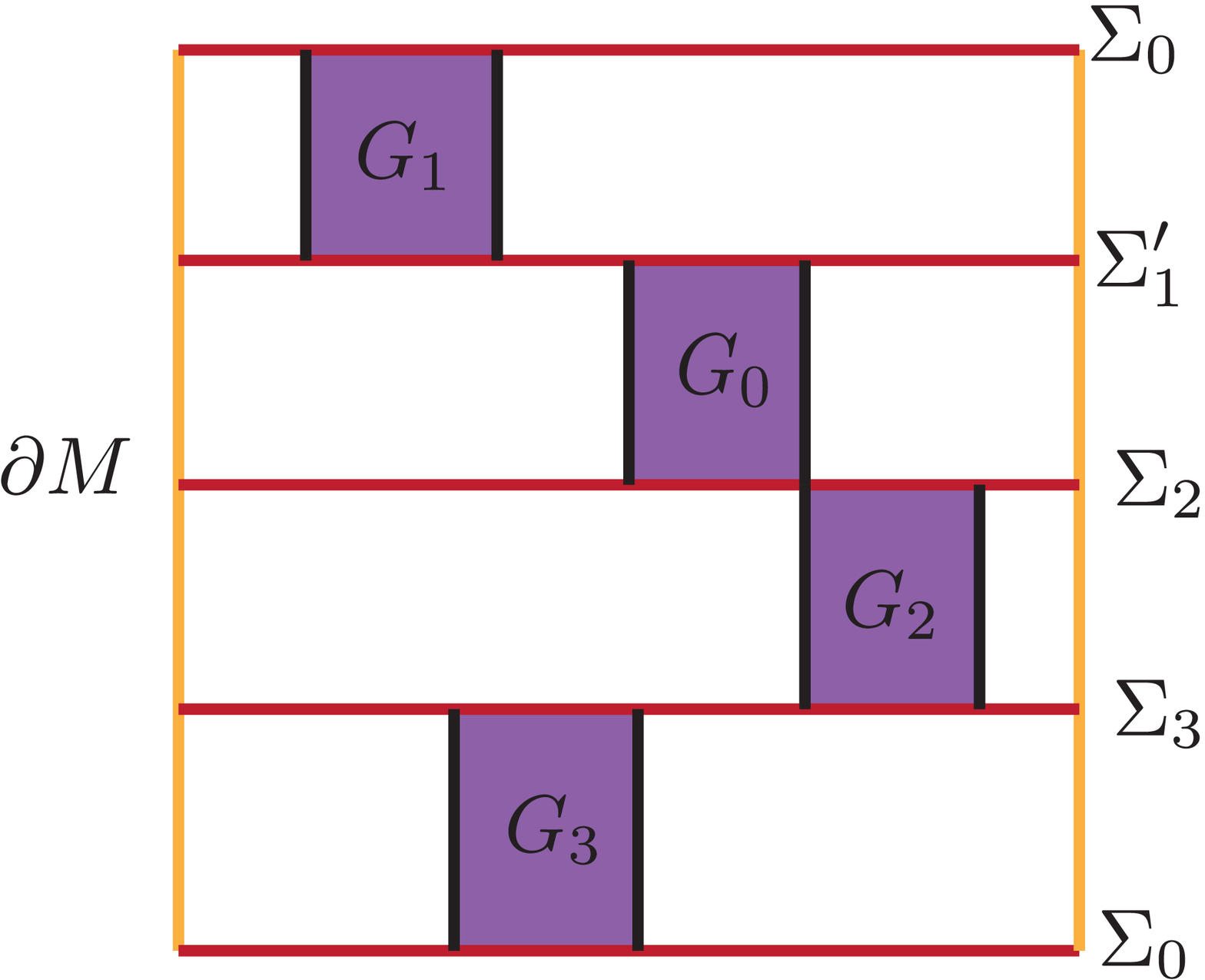}
\caption{}
\end{center}
\end{subfigure}%  

\caption{A schematic example of exchanging $G_0$ and $G_1$.} 
\end{figure}

\end{definition}

  For a {sutured bundle} structure with respect to $F$, we say a sutured guts component $G$ is \emph{next} to another $G'$ if $\Phi(G') =\Phi(G) + 1$, i.e., $G_{i+1}$ is \emph{next} to $G_i$ where we think of $G_k$ as $G_0$. We say that $G_i$ \emph{intersects} $G_{i+1}$ if $R_+(G_i)$ intersects $R_-(G_{i+1})$ in $\Sigma_{i+1}$ when they are isotoped to intersect minimally in $\Sigma_{i+1}$.

We define a map which relates $R_+$ and $R_-$ of a taut sutured manifold.

\begin{definition}\label{def:mapofR+R-}
Let $(M,R_\pm,\gamma)$ be a taut sutured manifold and $\{G_i | 0 \le i \le k-1\}$ are the sutured guts components of $M$. Let $X$ be a subsurface in $R_-$. We say $X$ is a \emph{step surface} if it either contains $R_-(G_i)$ or is disjoint with $R_-(G_i)$ for each $i$. Since $M \backslash G$ is a product sutured manifold, there is a natural projection from its $R_-$ to its $R_+$. We let $X'_1$ be the image of $X \backslash R_-(G)$ under this projection and glue $X'_1$ with $R_+(G_i)$ that are contained in $X$ to have $X'$. We define $\psi(X) \stackrel{\Delta}{=} X'$ and call it the \emph{ascending map} $\psi$ of $M$. Note that $\psi$ preserves the Euler characteristic and the genus.
\end{definition}

\begin{definition}
	Let $\{\sL_{j}(F) | 0 \le i \le k-1 \}$ be a sutured bundle structure for $M$ with respect to $F$. For a set of layers $\{\sL_{j}(F) | j\in J \}$, we denote $\uplus_{j\in J} \sL_{j}(F(z))$ as the sutured manifold obtained from gluing these $\sL_{j}(F)$ along $R_\pm$ if they coming from the same components of $F$.
\end{definition}

We have the following lemma to help producing some complicated pared guts.

\begin{lemma} \label{lem:intersects} Let $M$ be a hyperbolic manifold of finite volume. Given a sutured bundle structure for $M$ with respect to $F$, for an $i$ with $0 \le i \le k-1$, we can exchange $G_i$ with $G_{i+1},\ldots,G_{k-1},G_0,G_1,\ldots$ in sequence finitely many times so that in the new sutured bundle structure, $G_i$ intersects the guts component next to $G_i$.
\end{lemma}

\begin{proof}
	Without loss of generality, we assume $G_i$ is $G_0$. Suppose the lemma is not true for $G_0$. Since we can exchange $G_0$ with $G_1$, this means $R_+(G_0)$ is disjoint with $R_-(G_{1})$. Hence we can extend the sutured annuli of $G_0$ and $G_1$ to have $A$ so that its boundary lies on $\Sigma_0$ and $\Sigma_2$. If we cut $\sL_{0}(F)\uplus \sL_{1}(F) $ along $A$, we will have $G_0$, $G_1$ and some product sutured manifolds. Then we can replace $\Sigma_1$ by $\Sigma_2\backslash (R_+(G_0))$ glued by $R_-(G_0)$ to exchange $G_0$ and $G_1$. Since we can still exchange $G_0$ with $G_2,\ldots ,G_{k-1}$, we can extend the annuli sutures of $G_0$ so that its boundary lies on $\Sigma_0$, which means we have a union $A_0$ of product annuli in $M\spl \Sigma_0$ such that it separates $G_0$ and its complement.
	
	Let $\sL_{-1}(F)$ be the sutured manifold by gluing $\sL_{1}(F),\ldots,\sL_{k-1}(F)$, i.e., $M\spl \sL_{0}(F)$ and $G_{-1}$ be the guts of $\sL_{-1}(F)$. If we exchange $G_0$ with $G_1, \ldots ,G_{k-1}$ in sequence, we exchanging $G_0$ with $G_{-1}$ with respect to $\Sigma_0\cup \Sigma_1$.
	
	Let $\sL \stackrel{\Delta}{=} M\spl \Sigma_0$ and $\psi$ be the ascending map in Definition \ref{def:mapofR+R-}. Then the guts of $\sL$ are $G_0$ and $G_{-1}$. Let $R_{\pm,i}$ and $A_{i}$ be the $R_\pm$ and the sutured annuli of $G_i$ respectively for $i = 0,-1$. Then $\Sigma_1$ is $(\Sigma_0 \backslash R_{-,0})\cup (\psi(R_{-,0}) \cup A_{0}) = (\Sigma_0 \backslash R_{-,0})\cup (R_{+,0} \cup A_{0})$ up to isotopy.
	
	 We consider the cyclic covering $\tilde M$ of $M$ with respect to $\Sigma_0$. Let $\tilde \sL^0$ be a lifting of $\sL = M\spl \Sigma_0$ in $\tilde M$. We let $\tilde G_i^0$, $\tilde R^0_{\pm,i}$ and $\tilde A^0_{i}$ be the liftings of $G_i$, $R_{\pm,i}$ and $A_{i}$ in $\tilde \sL^0$ respectively for $i=0,-1$. Choose the deck transformation $\pi$ that $\pi$ maps the $R_-$ of $\tilde \sL^0$ to the $R_+$ of $\tilde \sL^0$. Here we think of all $R_-$ and $R_+$ are subsurfaces of liftings of $\Sigma_0$. Then we define $\tilde \sL^j$, $\tilde G_i^j$, $\tilde R^j_{\pm,i}$ and $\tilde A^j_{i}$ to be $\pi^j(\tilde \sL^0)$, $\pi^j(\tilde G_i^0)$, $\pi^j(\tilde R^0_{\pm,i})$ and $\pi^j(\tilde A^0_{i})$ respectively. Let $\tilde K^j$ be the $R_-$ of $\tilde \sL^j$ and $\tilde W^j$ be $\tilde \sL^j \backslash (\tilde G_0^j \cup \tilde G_{-1}^j)$. So $\{\tilde K^j\}$ are the liftings of $\Sigma_0$ and $\tilde W^j$ is a sutured product manifold.

\begin{definition}
	Let $X$ be a surface in $\tilde K^j$. We say $X$ is a \emph{step surface} if either contains $\tilde R^j_{-,i}$ or is disjoint with $\tilde R^j_{-,i}$ for $i=0,-1$.
\end{definition}

\begin{definition}\label{def:stairsurface}
	Let $X$ be a surface in $\tilde M$. We say $X$ is a \emph{stair surface} if it can be decomposed into the following parts: product annuli in $\tilde L^j$ for $j \in \Z$ and step surfaces in $\tilde K^j$ for $j \in \Z$. We call the \emph{height} of $X$ as the maximal $j$ such that $X \cap \tilde K^j$ is non-empty.
\end{definition}

We let $\psi_j$ be the ascending map for $\tilde \sL^j$ as in Definition \ref{def:mapofR+R-}.

If we exchange $G_0$ with $G_{-1}$, we actually replace $\Sigma_{1}$ with another surface $\Sigma'_{1} \stackrel{\Delta}{=} R_{-,0} \cup A_{0} \cup \psi(\Sigma_0 \backslash R_{-,0}) $. Let $\tilde \Sigma_0$ be $\tilde K^0$ and $\tilde \Sigma_{1}$ be $\tilde R^0_{-,0} \cup \tilde A_{0} \cup \psi_0(\tilde \Sigma_0 \backslash \tilde R^0_{-,0}) $. Then $\tilde \Sigma_0$ and $\tilde \Sigma_{1}$ are liftings of $\Sigma_0$ and $\Sigma'_1$. Furthermore, $\tilde \Sigma_0$ and $\pi(\tilde \Sigma_0)$ bounds a lifting of $\sL$ which contains $\tilde \Sigma_1$. Between $\tilde \Sigma_0$ and $\tilde \Sigma_1$ there is a guts component $\tilde G^0_0$ and between $\tilde \Sigma_1$ and $\pi(\tilde \Sigma_0)$ there is a guts component $\tilde G^0_{-1}$.

Since we can perform exchanging $G_0$ with $G_{-1}$ infinitely many times, we have a sequence of decomposition surfaces $\Sigma_0,\Sigma'_1,\Sigma'_2,\ldots$, where $\Sigma'_i,\Sigma'_{i+1}$ cut $M$ into two sutured manifolds each of which contains the sutured guts $G_0$ and $G_{-1}$ respectively.

Now, we perform a sequence of cut-and-paste operations to obtain a sequence of stair surfaces $\tilde \Sigma_0, \tilde \Sigma_1,\ldots$, where $\tilde \Sigma_{i+1}$ is obtained from $\tilde \Sigma_{i}$ by a cut-and-paste operation. These sequence has the property that $\tilde \Sigma_{i}$ is a lifting of $\Sigma'_i$ (here $\Sigma'_0 =\Sigma_0$), the height of $\Sigma'_i$ is $i$, between $\tilde \Sigma_{i+1}$ and $\pi(\tilde \Sigma_i)$ there is a guts component $\tilde G^0_0$ and between $\pi(\tilde \Sigma_i)$ and $\pi(\tilde \Sigma_{i+1})$ there is a guts component $\tilde G^{i+1}_{-1}$. See Figure \ref{fig:intersects}.

\begin{figure}[hbt]
\begin{center}
	 \includegraphics[width = 3 in]{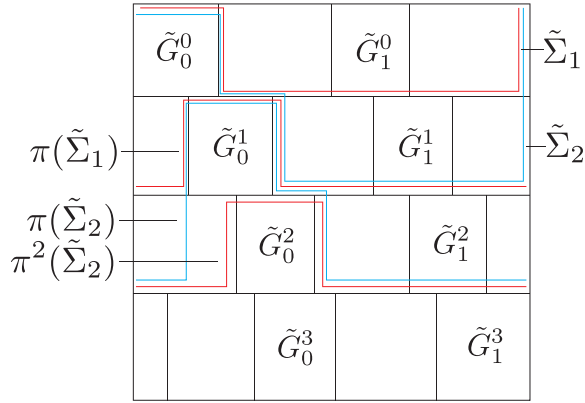}
  \caption{Stair surfaces.}\label{fig:intersects}

\end{center}
 \end{figure}

Let $\tilde \Sigma_{i}, \tilde \Sigma_{i+1}$ be in the sequence. We are going to construct $\tilde \Sigma_{i+2}$ from $\tilde \Sigma_{i+1}$. Since $G_0$ and $G_{-1}$ can be exchanged with respect to $ \Sigma'_{i}\cup \Sigma'_{i+1}$, $\tilde G^0_0$ and $\tilde G^{i+1}_{-1}$ can be exchanged. Hence the sutured annuli of $\tilde G^0_0$ can be extended to meet $\pi(\tilde \Sigma_{i+1})$ and the union $A_{i+1}$ of these extended annuli decompose the sutured manifold between $\tilde \Sigma_{i+1}$ and $\pi(\tilde \Sigma_{i+1})$ into two subset $G$ and $G'$, where $G$ is a sutured manifold equivalent to $\tilde G^0_0$. Furthermore, $R_+(G)$ is disjoint with $R_+(\tilde G^{i+1}_{-1})$ in $\pi(\tilde \Sigma_{i+1})$. We let $\tilde \Sigma_{i+2}$ be $(\tilde \Sigma_{i+1} \backslash R_-(G')) \cup A_{i+1} \cup R_+(G') $ and isotope it to be a stair surface. Since $R_+(G')$ is in $\pi(\tilde \Sigma_{i+1})$ and $R_+(\tilde G^{i+1}_{-1}) \subset R_+(G')$, $R_+(G')$ has height $i+2$. Since $\tilde \Sigma_{i+2} \cap \tilde K^{i+2} = R_+(G') \cap \tilde K^{i+2}$, $\tilde \Sigma_{i+2}$ has height ${i+2}$ too. Since $R_-(G)$ contains $R_-(\tilde G^0_0)$, $\tilde \Sigma_{i+2} \cap \tilde K^0$ is not empty.

So we have an infinite sequence $\tilde \Sigma_0, \tilde \Sigma_1,\ldots$ such that $\tilde \Sigma_i$ has height $i$ and $\tilde \Sigma_i \cap \tilde K^0 \ne \emptyset$. Since $\chi(\tilde \Sigma_i)= \chi(\tilde \Sigma_0) \stackrel{\Delta}{=} C$ is a constant, there is at most $|C|$ choices of $j$ such that $\tilde \Sigma_i \cap \tilde K^j$ is not a union of annuli. For any integer $n$, we take $i > (n+2)|C|$. Then there is an integer $m$ such that each of $\tilde \Sigma_i \cap \tilde K^j$ for $m \le j \le m+n$ is a union of annuli. Let $\tilde A_j$ be a product annulus in $\tilde \sL^j$ for $m \le j \le m+n$ such that $\tilde A_m, \tilde A_{m+1},\ldots, \tilde A_{m+n}$ can be connected to be an annulus. Because $\tilde \Sigma_i$ is a lifting of $\Sigma_i$, the images of $\tilde A_m, \tilde A_{m+1},\ldots, \tilde A_{m+n}$ under the projection to $M$ are disjoint product annuli in $\sL$. Because there are finitely many disjoint non-isotopic closed curves in a compact surface, when $n$ is big enough, there are $A_{m+d_1}$ and $A_{m+d_2}$ which are parallel. Then by gluing $A_{m+d_1},\ldots,A_{m+d_2-1}$, we have an essential torus in $M$ which is a contradiction to the atoroidality of $M$.

Therefore, we prove Lemma \ref{lem:intersects}.

\end{proof}

\begin{definition}[{\cite[Section 6.3]{agol2015certifying}}]
We call a taut sutured manifold $(M,R_\pm,\gamma)$ a \emph{library sutured manifold} if there is a taut surface $(\Sigma ,\partial \Sigma ) \subset (M,\eta(\gamma))$ such that $[\Sigma] = n[R_+] \in H_2(M,\eta(\gamma);\Z)$ for some $n \ge 0$, and the sutured manifold $M \backslash \Sigma$ is a book of $I$-bundles. We say that $\Sigma$ is a \emph{shelf surface} of $M$ and $M$ has height $(n+1)$.

As in Definition \ref{def:suturedbundle}, choose a base point $p$ close to $R_-$ and for a point $q$ in $M\backslash \Sigma$ we define $\Phi(q)$ to be the intersection number of a curve connecting $p,q$ and $\Sigma$ mod $k$. The \emph{potential} function $\Phi: M\backslash \Sigma \to \Z$ is well-defined and we can assign each component of $M\spl \Sigma$ to a number in $\Z$. Then for a number $i$ in $\Z$, we say the $i$-th \emph{layer} is the union of points in $M\spl \Sigma$ having the value of $\Phi$ as $i$. We think of the layer as a sutured manifold with $R_\pm$ as a union of components of $\Sigma$ and denote it as $\sL_i(\Sigma)$. We let $\Sigma_{i}$ be the $R_-$ part of $\sL_i(F)$ and call it a \emph{horizontal surface} of the library sutured manifold. 

	Given a library $M$ with a shelf surface $\Sigma$, We denote $G_i$ as the sutured guts component of $\sL_i(\Sigma)$. We call the union of $G_i$ the \emph{spine} of the sutured bundle. 		
\end{definition}
\begin{definition}\label{def:libroid}
In a 3-manifold $M$, we say that a taut surface $S \subset M$ is a \emph{libroid surface} if $M \backslash S$ is a library sutured manifold. An element $z$ in $H_2(M,\partial M)$ is \emph{libroid} if there is a libroid surface representing $z$. This is equivalent as saying $\Gamma(z)$ consists of solid tori and $T\times I$'s.

\end{definition}
\begin{definition}
	We use a diagram which we call \emph{library diagram} to represent the relation of the sutured guts in a library sutured manifold $M$ with a shelf surface $\Sigma$. Let $\sL_i(\Sigma)$ be the layers and $G_{i,j}$ be the guts components of $\sL_i(\Sigma)$. For each $i$, we place all $G_{i,j}$ vertically and if $G_{i,j}$ intersects $G_{i+1,j'}$, we draw an arrow from $G_{i,j}$ to $G_{i+1,j'}$. Worth to be mentioned, the non-existence of an arrow from $G_{i,j}$ to $G_{i+1,j'}$ does not indicate whether $G_{i,j}$ intersects $G_{i+1,j'}$.
	
	For example, in the following diagram, $G_{0,0}$ and $G_{0,1}$ are in the same layer and $G_{0,0}$ intersects $G_{1,0}$.
	\begin{equation*}
		\xymatrix{
G_{0,0} \ar[r] &   G_{1,0}     \\
G_{0,1} \ar[ur] & 
 }
\end{equation*}
	
In the same way, we can also use the library diagram to represent the relation of the sutured guts in a library bundle $M$ with respect to $F$.

\end{definition}

\begin{definition}[{\cite[Section 4.3]{agol2015certifying}}]We call a taut sutured manifold $(M,R_\pm,\gamma)$ a \emph{book of $I$-bundle} if the sutured guts of $M$ consists of solid tori and $T\times I$'s.

\end{definition}

We will frequently use the following definitions in finding the pared guts of a library sutured manifold.

\begin{definition} \label{def:essential}
Let $S$ be a compact surface and $X$ be a subsurface of $S$. We say $X$ is \emph{essential} if no component of $\partial X$ bounds a disk in $S$. For a subset $Y$ of $S$, we say $X$ is the \emph{minimal essential subsurface containing $Y$} if $X$ is essential and contains $Y$ with the maximal Euler characteristic, the minimal genus and then the maximal number of boundary components among all such surfaces. %The reason why we require the maximal number of boundary components is because we want to separate parallel closed curves.
\end{definition}

\begin{definition}
	A \emph{pared sutured manifold} $(M,P,R_\pm,\gamma)$ is a compact oriented manifold with two sets $P$ and $\gamma$ consisting of pairwise disjoint annuli and tori such that $\gamma \subset P \subset \partial M$. Furthermore, $(M,R_\pm,\gamma)$ is a sutured manifold and $(M,P)$ is a pared manifold.
\end{definition}

The following lemma is about the pared guts of library sutured manifolds.

\begin{lemma} \label{lem:gutsoflibrary}
	 Let $(M,R_\pm,\gamma)$ be a {library sutured manifold} of height $n$ with spine consists of 2-STs, 4-STs and $T\times I$'s. Furthermore, if we think of $\gamma$ as the pared locus, $M$ is a pared manifold. Let $\sL_i$ be the $i$-th layers and $G_i$ be the sutured guts of $\sL_i$ for $0\le i \le n-1$. Let $\Sigma_0, \ldots, \Sigma_n$ be the horizontal surfaces and $\psi_i$ be the ascending map for $\sL_i$ as in Definition \ref{def:mapofR+R-}. Assume that components of $R_-(G_i)$ (components of $R_+(G_i)$) are not isotopic to each other in $\Sigma_i$ ($\Sigma_{i+1}$) for $0\le i \le n-1$.
	 \begin{enumerate}
	 	\item $n=2$. Let $X$ be the union of negative Euler characteristic components of minimal essential subsurface in $\Sigma_1$ containing $R_+(G_0)$ and $R_-(G_1)$. Then the pared guts of $M$ is a pared sutured manifold $(N,P,R_\pm,\gamma)$ where $N$ is $X\times I$, $\gamma = \partial \times I$, $R_- = X\times \{0\}$ and the pared locus $P$ consists of $\gamma$, $ (R_+(G_0) \cap X)$ and $(R_-(G_1) \cap X)$.
	 	\item $n=3$. Suppose that $R_-$ of each component of $G_1$ intersects $R_+(G_0)$ and $R_+$ of each component of $G_1$ intersects $R_-(G_2)$. Let $Z$ be the minimal essential subsurface in $\Sigma_1$ such that $R_+(G_0) \cup R_-(G_1) \subset Z$ and $R_+(G_1)\cup R_-(G_2) \subset \psi_1(Z)$. Take $X$ as the union of negative Euler characteristic components of $Z$. Then the pared guts of $M$ is a pared sutured manifold $(N,P,R_\pm,\gamma)$ where $R_-$ is $X$, $R_+ =\psi_1(X)$, $\gamma$ is the union of product annuli connecting $\partial R_+$ and $\partial R_-$ and tori from $T\times I$ components of $G_1$, and $P$ is $\gamma \cup (R_+(G_0) \cap R_-)\cup (R_-(G_1) \cap R_+)$.
	 	\item $n=4$. Suppose that for each $i$, $G_i$ is a 4-ST and $G_i$ intersects $G_{i+1}$. Let $Z$ be the minimal essential subsurface in $\Sigma_2$ such that $R_+(G_1) \cup R_-(G_2) \subset Z$, $R_+(G_2)\cup R_-(G_3) \subset \psi_2(Z)$ and $R_+(G_0)\cup R_-(G_1) \subset \psi_1^{-1}(Z)$. Take $X$ as the union of negative Euler characteristic components of $Z$. The pared guts of $M$ is a pared sutured manifold $(N,P,R_\pm,\gamma)$ where $N$ is a product manifolds $X\times I$ glued by a thickened annulus on each side via $R_+(G_1)$ and $R_-(G_2)$ respectively and with annular cusps on each side coming from $R_+(G_0)$ and $R_-(G_3)$. 
	  \end{enumerate}

\end{lemma}
\begin{proof} Let $X$ be as in the statements of the lemma.

 Because $G_0$ consists of solid tori and $T\times I$'s, the union $R_+(G_0)\cup A(G_0)$ of $R_+$ and sutured annuli of $G_0$ is a union of essential annuli and so is $R_-(G_{n-1})\cup A(G_{n-1})$. We first cut $M$ along $R_+(G_0)\cup A(G_0)$ and $R_-(G_{n-1})\cup A(G_{n-1})$ to have a new pared sutured manifold $(M',P',R'_\pm,\gamma')$ where $M' = \uplus_{i=1}^{n-1} \sL_i$ (if $i=2$, $M'$ is a product manifold $\Sigma_1 \times I$), $\gamma' =\gamma$, and $P' = \gamma \cup R_+(G_0) \cup R_-(G_{n-1})$. Then we cut along product annuli coming from $\partial X$ to get the pared sutured manifold $(N,P,R_\pm,\gamma)$ which is stated in Lemma \ref{lem:gutsoflibrary} and some product sutured manifolds.
 
 It suffices to show that $(N,P)$ is acylindrical. Let $A$ be an essential annulus with $\partial A \subset P$. If $A$ is a product annulus and by the minimality of $X$, $A$ intersects $ R_+(G_0) \cup R_-(G_{n-1})$, which is impossible. So we assume both of the boundary curves of $A$ lie on $R_-(N)$ or $R_+(N)$. Without loss of generality, we assume that both of them lie on $R_-(N)$. 
 
In Case (1) for $n=2$, $(N,R_\pm,\gamma)$ is a product sutured manifold. Because the components of $R_+(G_0)$ are not parallel in $\Sigma_1$, any essential annulus is a product annulus. 
 
In Case (2) for $n=3$, do a cut-and-paste surgery for $\Sigma_1$ along $A$ to have $S'$. Then $S'$ is a taut surface homologous to $\Sigma_1$ and the guts $G$ between $S'$ and $\Sigma_1$ is a 4-ST with the double of $A$ as $R_+$. Since the horizontally prime guts of $(N,R_\pm,\gamma)$ is $G_1$, by \cite[Corollary 3.20]{AZ1}, $G$ is a component of $G_1$ and $\partial A$ is the union of the cores of $R_-(G)$ up to isotopy. Because $R_-$ of each component of $G_1$ intersects $R_+(G_0)$, we cannot have $\partial A \subset P$, a contradiction.
 
In Case (3) for $n=4$, we isotope $A$ so that it intersects $X$ minimally. Then $X$ decomposes $A$ into pieces of annuli. Let $A_0$ be a component of $A\backslash X$ such that $A_0$ is not a product annulus. By the same argument in the previous paragraph, $\partial A_0$ is the union of the cores of $R_-(G_i)$ up to isotopy for $i=1$ or 2. Since $R_+(G_1)$ intersects $R_-(G_2)$, if $\partial A_0$ is the union of the cores of $R_-(G_2)$, $A_0$ cannot be connected by annuli in $\sL_1$. So we know that $A_0$ is in $\sL_1$. Since $R_+(G_0)$ intersects $R_-(G_1)$, we cannot have $\partial A \subset P$, a contradiction.

\end{proof}

\section{3-Cusped Hyperbolic Manifolds} \label{sec:volume:3-cusped}

 Let $M$ be 3-cusped orientable hyperbolic manifold of finite volume. By applying the following ``half lives, half dies" lemma, we know that $H_2(M, \partial M)$ is at least rank 3.
 
 \begin{lemma}[{\cite[Lemma 3.5]{hatcher2000notes}}] \label{lem:halflives}
 	If $M$ is a compact orientable 3-manifold then the image of the boundary
map $H_2(M, \partial M)\to H_1(\partial M)$ has rank equal to one half the rank of $H_1(\partial M)$.
 \end{lemma}

 \begin{lemma} \label{lem:sequence}
 
	Let $M$ be a 3-cusped orientable hyperbolic manifold and $z$ is a vertex of the Thurston sphere of $M$. Suppose $\Gamma(z)$ consists of solid tori and $T\times I$'s. Then $\Gamma(z)$ is at least depth 2. 
	
	Furthermore, let $e$ be an open Thurston face whose closure contains $z$ and $f$ be an open Thurston face whose closure contains $e$. Then for any element $y$ in $e$ and $y'$ in $f$, we can take $u$ in the open segment $(z,y)$ and $v$ in the open segment $(u,y')$ such that $\Gamma(z)$ has a sequence of 2 sutured decomposition
\[
	\Gamma(z) \stackrel{S_1}{\overline \leadsto} \Gamma(u)  \stackrel{S_2}{\overline \leadsto} \Gamma(v).
	\]

where $S_1$ and $S_2$ are properly norm-minimizing surfaces in $M$. Furthermore, $[S_1]$ is in $(z,y)$ and $[S_2]$ is in $(u,y')$.

\end{lemma}
\begin{proof}
	By \cite[Theorem 1.4]{AZ1}, we can take an element $w$ in $(z,y)$ such that $\Gamma(z)$ can be decomposed along a properly norm-minimizing surface $S_1$. The resulting sutured manifold is actually a union of the sutured guts of a surface $X$ representing a multiple of $u$ in $(z,y)$ and some product sutured manifolds. Since each component of $\Gamma(z)$ either remains the same or is decomposed to product sutured manifolds, the sutured guts $\Gamma(M,X)$ should consist of some components of $\Gamma(z)$. Furthermore, by Proposition \ref{prop:gutcomponentinhyperbolic}, there are no other taut surfaces in $\Gamma(M,X)$ not parallel to the boundary. Hence $\Gamma(M,X)$ is actually $\Gamma(u)$ and we have:
	\[\Gamma(z) \stackrel{S_1}{\overline \leadsto} \Gamma(u).\]
	
	Similarly, we can take an element $v$ in $(u,y)$ such that $\Gamma(z)$ has a sequence of 2 sutured decomposition
	\[
	\Gamma(z) \stackrel{S_1}{\overline \leadsto} \Gamma(u)  \stackrel{S_2}{\overline \leadsto} \Gamma(v).
	\]
	
	\end{proof}

	\begin{definition}\label{def:vanish}
		Let $M$ be an orientable hyperbolic 3-manifold with 3 cusps and $z$ be a class in $H_2(M,\partial M)$. There is a natural boundary map $\partial : H_2(M,\partial M) \to H_1(\partial M)$ coming from the long exact sequence for $(M,\partial M)$. We say $z$ \emph{vanishes} on a boundary component $P$ of $M$ if $\partial z$ is 0 in the $H_1(P)$ summand of $H_1(\partial M)$.
	\end{definition}

From Lemma \ref{lem:sequence}, we can deduce the following lemma.

\begin{lemma} \label{lem:suturenontrivial}
	Let $M$ be an orientable hyperbolic 3-manifold of finite volume with 3 cusps and $z$ be a vertex of the Thurston sphere and also a libroid class. If there is no class in $H_2(M,\partial M)$ that vanishes on two boundary components, then there exists two sutured guts components $G_0$ and $G_1$ of $\Gamma(z)$ such that each component of $s(G_0)$ and $s(G_1)$ is homologically non-trivial and a component of $s(G_0)$ is not homologous to any multiple of a component of $s(G_1)$ in $H_1(M)$.
\end{lemma}
\begin{proof}
	Otherwise, we suppose there exists a nontrivial 1-homology class $c $ in $H_1(M)$ such that for each sutured guts component $G$ of $\Gamma(z)$, any component of $s(G)$ is either homologically trivial or homologous to a multiple of $c$. We denote $\Gamma_1$ and $\Gamma_2$ as the union of guts components of the first type and the second type respectively.
	
	By Poincare duality, there exists a class $y$ in $H_2(M,\partial M)$ such that the algebraic intersection number $\langle y,c\rangle \ne 0.$ Because of \cite[Theorem 1.3]{AZ1}, $z$ is trivial in $H_2(\Gamma(z),\partial \Gamma(z))$. So $z$ is not a multiple of $y$. 
	
	Let $e$ and $e'$ be two 1-codimensional open Thurston face whose closure contains $z$ such that they are not in the opposite direction. Then the subspace spanned by $e$ and $e'$ is $H_2(M,\partial M)$. Since $\langle y,c\rangle \ne 0$, there is an element $w$ in one of $e$ and $e'$, say $e$, such that $\langle w,c \rangle \ne 0$.

	Let $f$ be an open Thurston face whose closure contains $e$. By Lemma \ref{lem:sequence}, we can take $u$ in the open segment $(z,w)$ and $v$ in $f$ such that $\Gamma(z)$ has a sequence of 2 sutured decomposition
\[
	\Gamma(z) \stackrel{S_1}{\overline \leadsto} \Gamma(u)  \stackrel{S_2}{\overline \leadsto} \Gamma(v).
	\]

where $S_1$ and $S_2$ are properly norm-minimizing surfaces in $M$. Furthermore, $[S_1]$ is in $(z,w)$ and $[S_2]$ is in $f$. 

Since $\langle [S_1], c \rangle$ is a multiple of $\langle w,c\rangle$, $[S_1]$ is nontrivial in $H_2(G,\partial G)$ for each guts component $G$ of $\Gamma_2$. So by \cite[Section 4]{Ga1} or \cite[Theorem 3.5]{FK}, $S_1$ nontrivially decompose each guts component $G$ of $\Gamma_2$. Furthermore, since $S_1$ intersects $c$ nontrivially and each component of $\Gamma(z)$ is a 2-ST, a 4-ST or a $T\times I$, $S_1$ decomposes $\Gamma_2$ to a union of product sutured manifolds. So $\Gamma(u)$ is a non-empty subset of $\Gamma_1$. Since for each  guts component $G'$ of $\Gamma(u)=\Gamma_1$, each of its sutures is homologically trivial, $S_2$ should meet $\partial G'$ in curves parallel to the sutures. Hence if $G'$ is a 2-ST or 4-ST, $S_2 \cap G'$ is a trivial decomposition surface. Since $\Gamma(u)  \stackrel{S_2}{\overline \leadsto} \Gamma(v)$ is a nontrivial sutured decomposition, there is a $G' =T\times I$ in $\Gamma_1$. By the definition of $\Gamma_1$, its sutures are homologically trivial in $H_1(M)$. Let $\eta$ be a closed curve on $\partial M \cap \partial (T \times I)$ such that $\eta$ is isotopic to the sutures of $G'$. Then $[\eta]$ is 0 in $H_1(M)$. Let $b$ be $[\eta]$ in $H_1(\partial M)$. From the long exact sequence:
\[
\cdots \to H_2(M,\partial M) \stackrel{\partial}{\to} H_1(\partial M) \to H_1(M) \to \cdots,
\]
since $b$ is mapped to $0$ in $H_1(M)$, we know that there is a class $z'$ in $H_2(M,\partial M)$ such that $\partial z' = b$ and hence $z'$ vanishes on the two boundary components of $M$ that $\eta$ is disjoint with, which is a contradiction to the assumption.
\end{proof}

\section{Volume of Library Bundles} \label{sec:volume:volumeoflibrary}

Now we start estimating the volume of library bundles and proving Theorem \ref{thm:2TIor4ST}. 

For a facet surface $F(z)$ in $M$, we have a sutured bundle structure for $M$ with $S_i$ as horizontal surfaces. In this section, we drop $\Sigma_{i+1}$ if $\Sigma_i$ is parallel to $\Sigma_{i+1}$ and think of the union of the rest of $\Sigma_i$ as the facet surface $F(z)$. So $\Sigma_i$ is non-isotopic to each other.

Let $M$ and $z$ be as in Theorem \ref{thm:2TIor4ST}. We use the notation for library bundles in Section \ref{sec:Library Bundles} so that we have horizontal surfaces $\Sigma_{i}$ and layers $\sL_i(F(z))$.

\subsection{Only two sutured guts components in $\Gamma(z)$ with one of them as $T \times I$}

We first consider the case when $\Gamma(z)$ consists of two $T\times I$'s or 4-STs and at least one of them is a $T\times I$.

\begin{lemma} \label{lem:2TI4ST}
	Let $M$ be an orientable hyperbolic 3-manifold with 3 cusps and $z$ be a class in $H_2(M,\partial M)$. If one of the following holds:
	\begin{enumerate}
		\item $\Gamma(z)$ consists of two $T\times I$'s, or
		\item $z$ is a vertex element of the Thurston sphere with $\Gamma(z)$ consisting of a $T\times I$ and a 4-ST, and there is no class in $H_2(M,\partial M)$ that vanishes on two boundary components of $M$,
	\end{enumerate}
	then $\Vol(M) \ge 2V_8.$ 
\end{lemma}
\begin{remark}
	The inequality is sharp for the Borromean rings complement in this case. See \cite[Example 7.3]{AZ1}.
\end{remark}

\begin{proof}

 We are going to find a nice essential surface that can give us a good estimation (cf. \cite[Theorem 3.4]{agol2010minimal}).

 We let a guts component $G_0$ be a $T\times I$ and another one $G_1$ be a 4-ST or a $T\times I$. By Lemma \ref{lem:intersects}, after exchanging $G_0$ with $G_1$ finitely many times, we can assume that $G_0$ intersects $G_1$ and still think of the new facet surface as $F(z)$. 
 
  Note that $\partial G_0 \cap \partial M$ is a torus. We do an annular compression for $\Sigma_0$ along a vertical annulus from $\Sigma_0$ to $\partial G_0 \cap \partial M$. Denote $Y_1$ as the new surface. Note that $Y_1$ is not a sutured decomposition surface. However $Y_1$ is essential and, by Lemma \ref{lem:pared}, $ M \spl Y_1$ is a pared manifold $(N_1,P_1)$. Then $M \spl Y_1$ is obtained from $\sL_1(F(z))$ by adding 2 pared annuli whose cores are denoted as $c_0$ and $e_1$ on $R_-$ and $R_+$ of $\sL_1(F(z))$ respectively. Here and in the following we use a closed curve to represent its regular neighborhood in a surface. See Figure \ref{fig:TI1}.
 \begin{figure}[hbt]
	\includegraphics[height = 1.5 in]{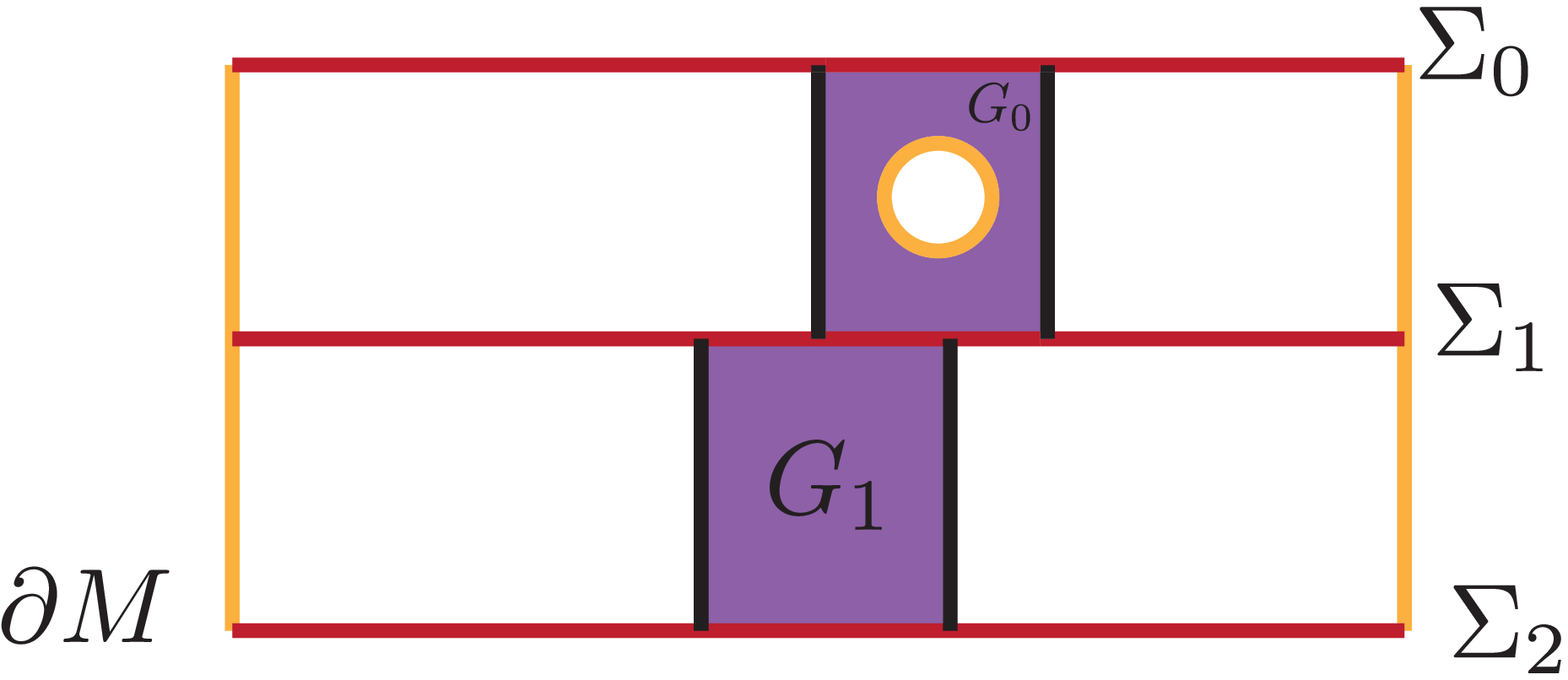}
	\includegraphics[height = 1.5 in]{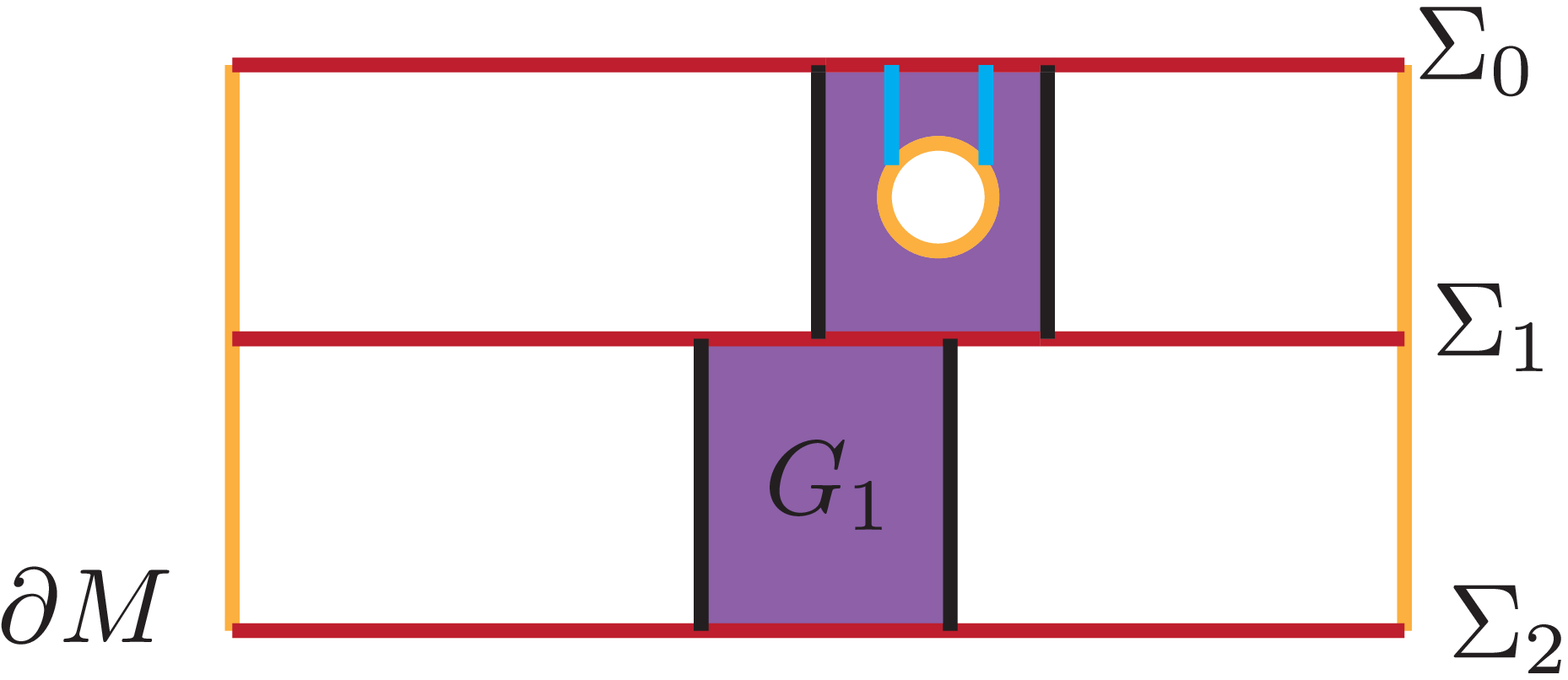}
	\includegraphics[height = 1.5 in]{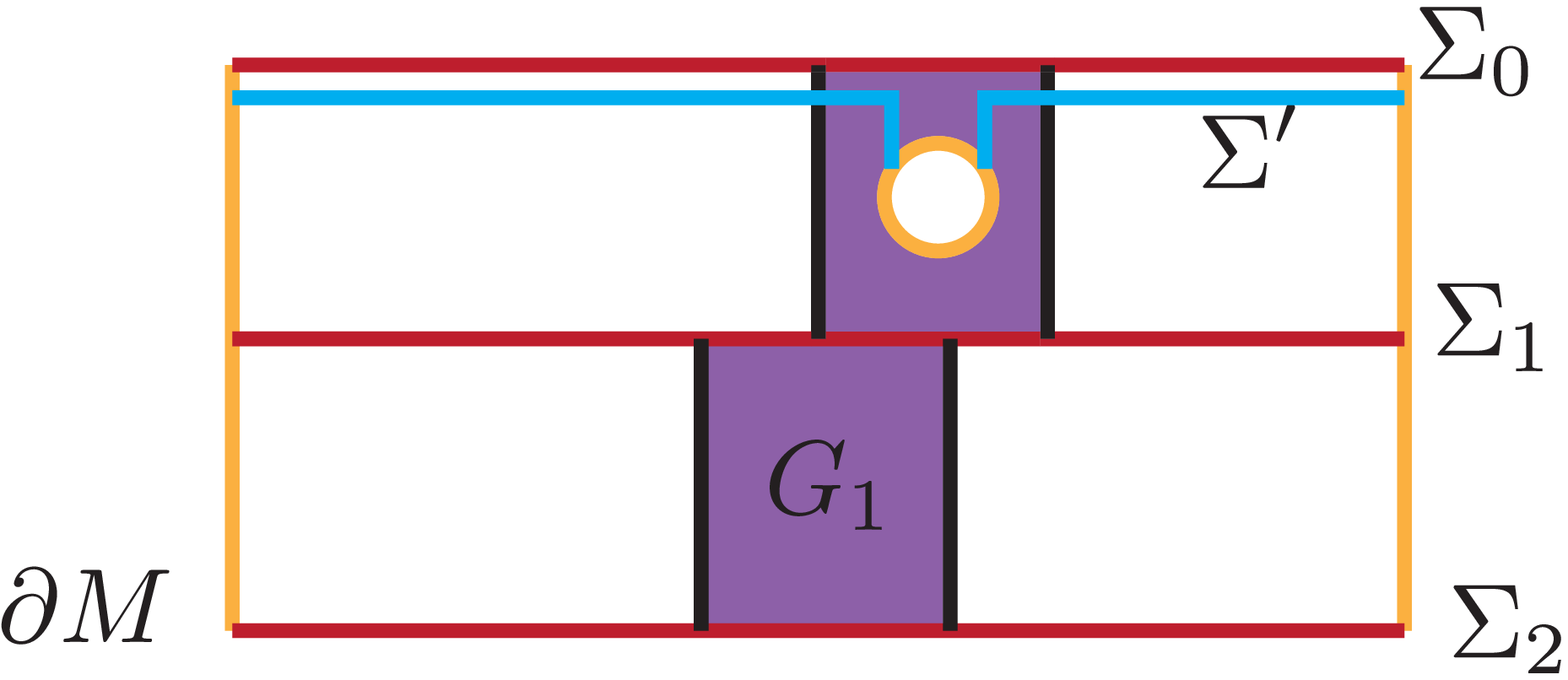}
	\includegraphics[height = 1.5 in]{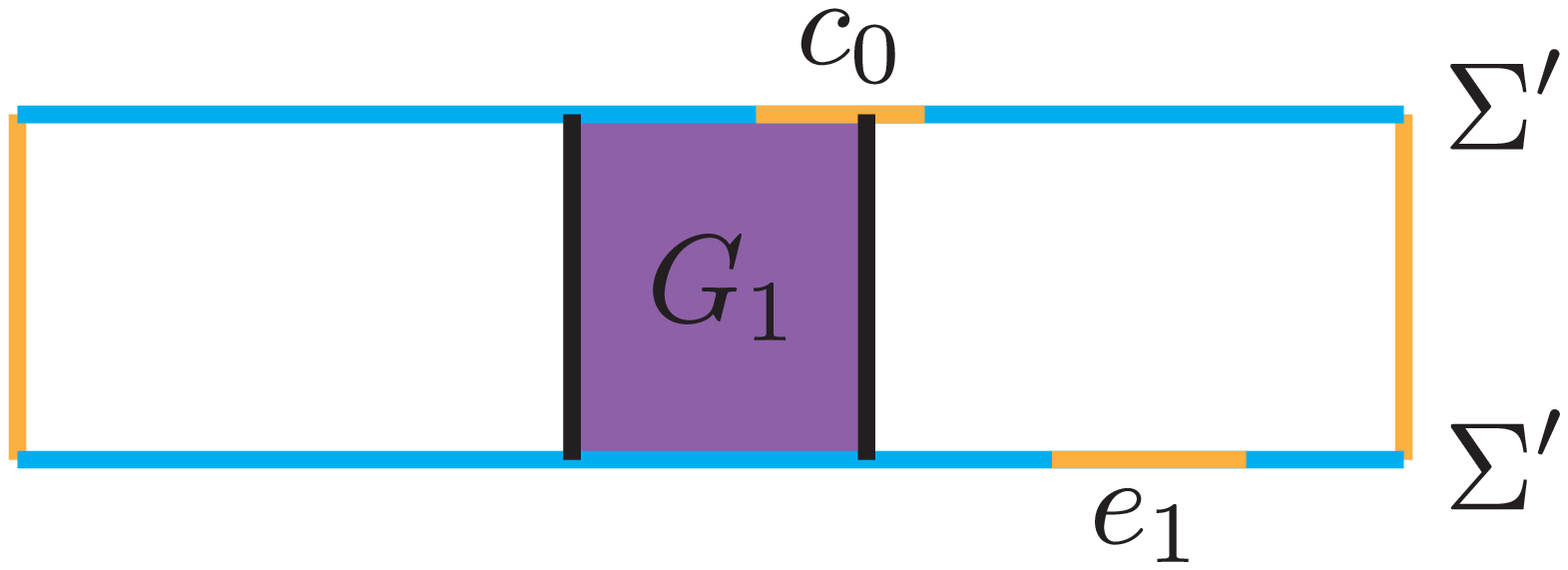}
\caption{An annular compression for $\Sigma_0$} \label{fig:TI1}
\end{figure}
 
%If $e_1$ is disjoint with $R_+(G_1)$ and the vertical annulus $A_1$ connecting $e_1$ and $Y_1$ is disjoint with $c_0$, we can do another annular compression for $Y_1$ along a $A_1$. Denote $Y_2$ as the new surface and then $M\spl Y_2$ comes from removing the pared annulus $e_1$ of $M\spl Y_1$, cutting along $A$ to have two pared annuli and adding a new pared annulus $e_2 $ on the same side of $e_1$. Then the underlying manifold is $X_2\times I$ drilled out a curve $\zeta_1$, where $X_2$ is a subsurface of $\Sigma_1$ by cutting $\Sigma_1$ along $e_1$.
%TODO (figure) 

Because $G_0$ intersects $G_1$, $c_0$ intersects $R_-(G_1)$. If $e_1$ is disjoint with $R_+(G_1)$, there is an annulus $A_1$ such that one of its boundary component is $e_1$ and the other boundary component lies on $Y_1$. If $\partial A_1 \cap Y_1$ is disjoint with $c_0$, we can do another annular compression for $Y_1$ along $A_1$ to have a new surface $Y_2$. Then the new pared manifold $(N_2,P_2) \stackrel{\Delta}{=} M\spl Y_2$ comes from removing the annular annulus $e_1$ of $N_1$, cutting along $A_1$ and adding a new pared annulus $e_2 $ on the same side of $e_1$. Here we think of the annuli from $A_1$ as pared annuli. Then the underlying manifold $N_2$ can be obtained from gluing $G_1$ and a product sutured manifold $W_2$ along some components of the sutured annuli. See Figure \ref{fig:TI2}.

\begin{figure}[hbt]
	\includegraphics[height = 1.5 in]{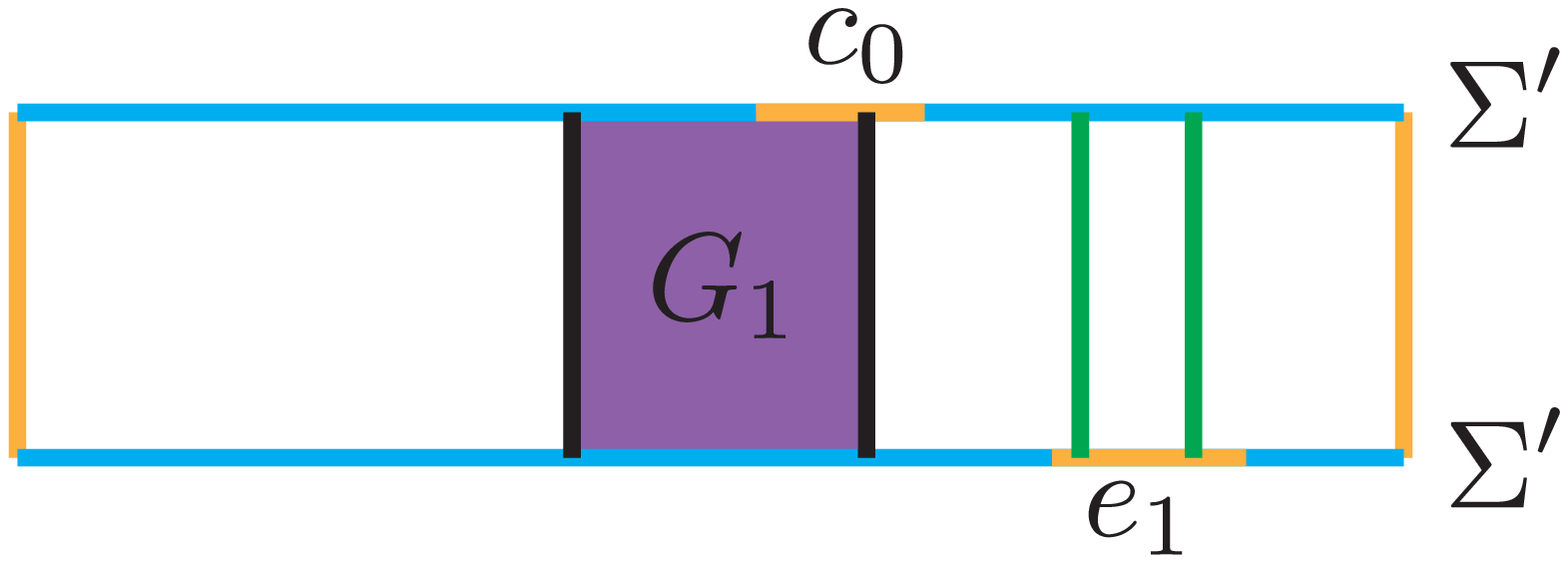} 
	\includegraphics[height = 1.5 in]{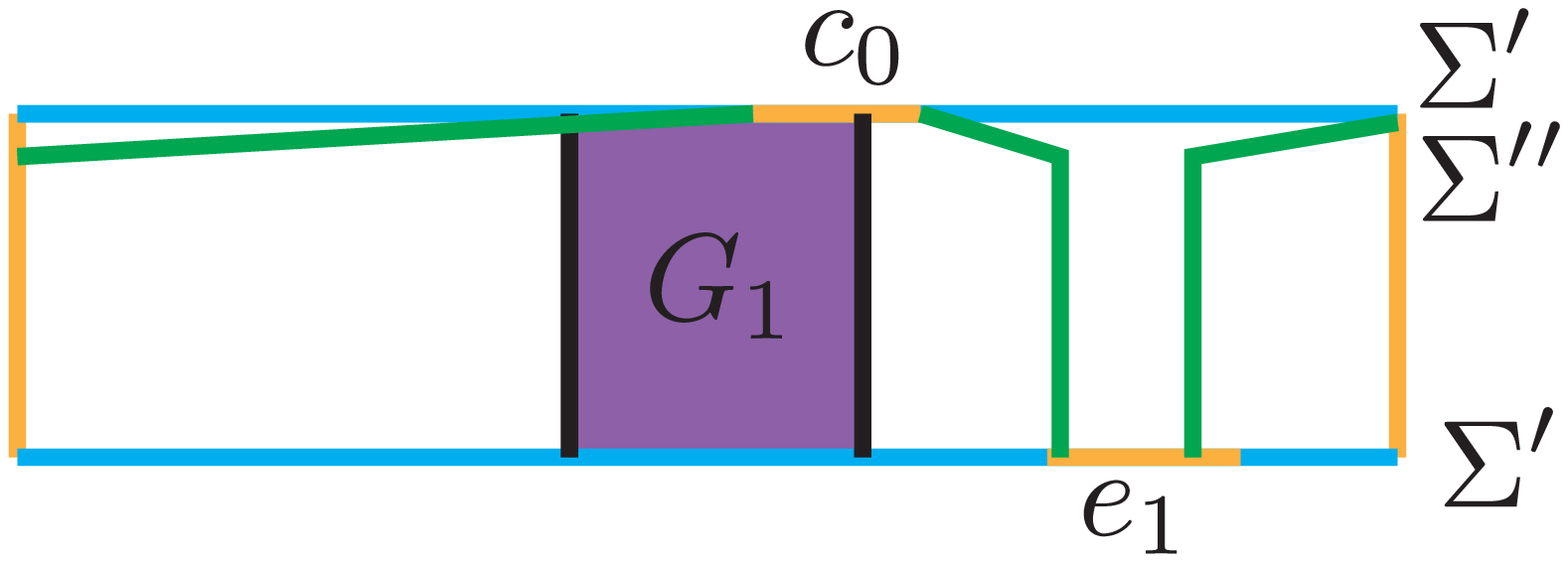}
	\includegraphics[height = 1.5 in]{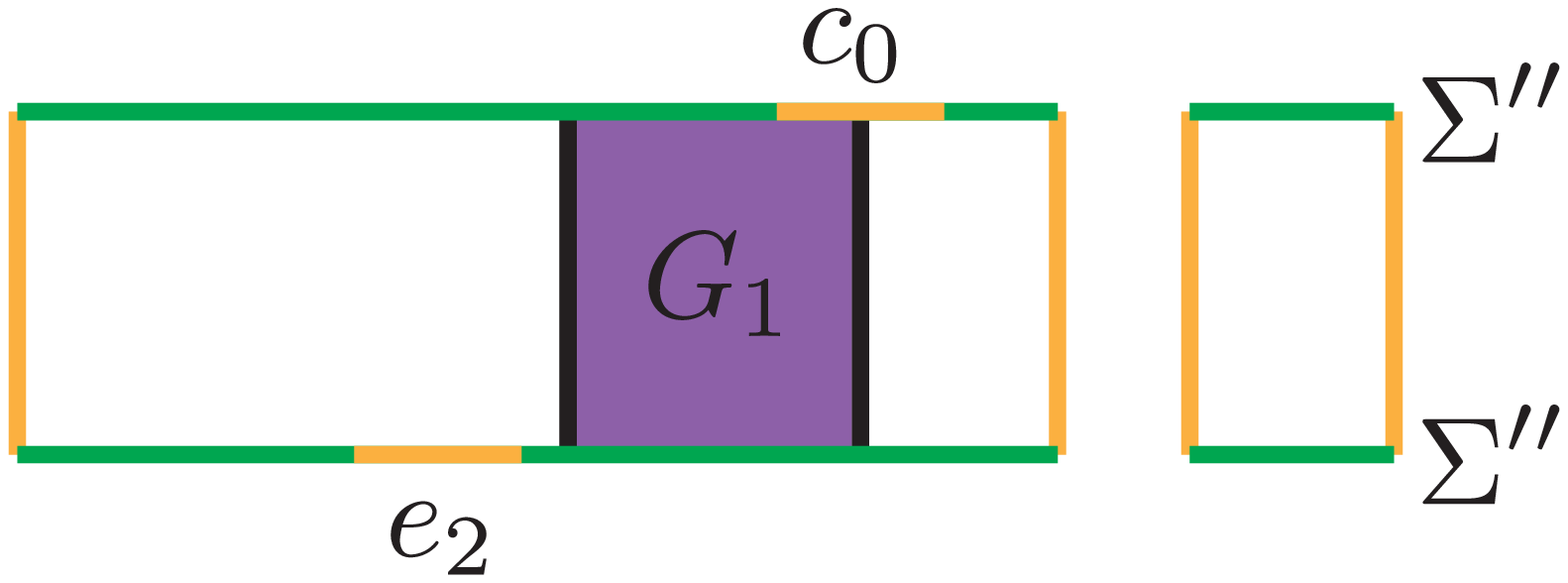} 
	\caption{Another annular compression for $\Sigma'$} \label{fig:TI2}
\end{figure}

Suppose we have $Y_i$, $(N_i,P_i)$ and $e_i$ such that $M\spl Y_i$ is a pared sutured manifold $(N_i,P_i, R_\pm,\gamma)$. $(N_i,R_\pm,\gamma)$ is obtained from gluing $G_1$ and a product sutured manifold $W_i$ along some components of the sutured annuli. The annuli components of the pared loci $P_i$ are $c_0$, $e_i$ and $\gamma$. Furthermore, they have Property (D): $e_i$ is disjoint with $R_+(G_1)$ and the vertical annulus $A_i$ which connects $e_i$ and $Y_i$ is disjoint with $c_0$. Then we do an annular compression for $Y_i$ along $A_i$ to have $Y_{i+1}$. The new pared manifold $(N_{i+1},P_{i+1}) \stackrel{\Delta}{=} M\spl Y_{i+1}$ comes from removing the annular cusp $e_i$ of $N_i$, cutting along $A_i$ and adding a new pared annulus $e_{i+1} $ on the same side of $e_i$. Here we think of the annuli from $A_i$ as pared annuli. $N_{i+1}$ can be obtained from gluing $G_i$ and a product sutured manifold $W_{i+1}$ along some components of the sutured annuli. 

If $(N_{i+1},P_{i+1})$ and $e_{i+1}$ have Property (D), we can repeat this operation. So we have a sequence of $Y_i$, $(N_i,P_i)$ and $e_i$. Since $\chi(Y_{i+1}) = \chi(Y_{i})$ and $Y_{i+1}$ has two more boundary components than $Y_{i}$, the number of annular compressions is finite. Thus at some stage we have $Y_n$, $(N_n,P_n)$ and $e_n$ such that either $e_n$ intersects $R_+(G_1)$ or the vertical annulus $A_n$ which connects $e_n$ and $Y_n$ intersects $c_0$.

If $G_1$ is a $T\times I$, $M\spl Y_n$ is obtained from a product manifold $X_n \times I$ by drilling out a curve $c_1 $ of $X_n\times \{1/2\}$ and adding pared annuli $c_0$ and $e_n$ on $X_n\times \{0\}$ and $X_n\times \{1\}$ respectively. Furthermore, $c_0$ intersects $c_1$ and $e_n$ intersects $c_0$ or $c_1$ under the projection to $X_n$. Let $X$ be the minimal essential subsurface containing $c_0$, $c_1$ and $e_n$. Then the pared guts $(G,P)$ of $M\spl \Sigma_n$ is an $I$-bundles $X \times I$ with two pared annuli $c_0$ and $e_n$ on each side and drilled out a curve $c_1$ in $X\times\{1/2\}$. Furthermore, $\partial X \times I$ are also pared annuli. Because $M$ is 3-cusped and there are only two sutured guts $T\times I$ in $\Gamma(z)$, $X_n$ has non-empty boundary and hence $X$ also has boundary. Therefore there are at least 3 annular cusps and 1 toral cusp in this gut. By Theorem \ref{thm:4cusps}, the volume of the pared guts $G$ is at least $2V_8$ and hence by Theorem \ref{thm:AST}, $\Vol(M) \ge 2V_8$.

If $G_1$ is a 4-ST, $M\spl Y_n$ is obtained from $X_n \times I$ by gluing a thickened annulus to $X_n \times \{1\}$ via two non-parallel closed curves $l_1,l_2$ and have 2 pared annuli whose cores are $c_0$ and $e_n$ on each side respectively. Furthermore, either $e_n$ intersects $l_1,l_2$, or the annulus connecting $e_n$ to $X_n \times \{0\}$ intersect $c_0$. Because $M$ is 3-cusped and there is only a sutured guts component $T\times I$ in $\Gamma(z)$, $Y_n$ has boundary and hence $X_n$ also has boundary. As in Lemma \ref{lem:gutsoflibrary}, by taking the pared guts of this pared manifold, we have a pared gut $(G,P)$ which is a pared sutured manifold coming from an $I$-bundles $X \times I$ glued by a thickened annulus via $l_1,l_2$. Note that $X$ contains $c_0,e_n,l_1,l_2$. Since $G$ is acylindrical, each component of $X$ has negative Euler characteristic. If $\chi(X) = -1$, $X$ is connected. Because there is a pair of curves $(c_0,l_i)$ in $X$ intersecting nontrivially, the genus of $X$ is at least 1 and hence $X$ is a punctured torus. Because $l_1,l_2$ are two disjoint non-parallel curves in $X$, we can assume $l_1$ intersects $c_0$ and then $l_2$ has to be parallel to the boundary of $X$. Hence $l_2$ is homologically trivial in $H_1(M)$. Because $l_2$ is isotopic to each component of $s(G_1)$, each component of $s(G_1)$ is homologically trivial in $H_1(M)$ which is a contradiction to Lemma \ref{lem:suturenontrivial}. So we have $\chi(X) \le -2$. Hence by Theorem \ref{thm:AST}, the volume of the guts is at least $2V_8$ too.  
\end{proof}

\subsection{More than two sutured guts components in $\Gamma(z)$}
 
 Next we consider the case when $\Gamma(z)$ consists of more than two solid tori and $T\times I$'s.

 \begin{lemma} \label{lem:atleast3}
	Let $M$ be an orientable hyperbolic 3-manifold with 3 cusps and $z$ be a vertex element of the Thurston sphere such that $\Gamma(z)$ consists of more than 2 solid tori and $T\times I$'s. Suppose there exists two 4-STs or $T\times I$'s $G$ and $G'$ in $\Gamma(z)$ and furthermore, one of the following holds:
	\begin{enumerate}
		\item each suture of $G$ and $G'$ is homologically nontrivial in in $H_1(M)$ and each suture of $G$ is not isotopic to any suture of $G'$, or
		\item $G$ and $G'$ are two $T\times I$'s, or
		\item $G$ is a $T\times I$ and $\Gamma(z)$ cannot be divided into two sets such that the sutures of all sutured guts components in each set are homologous up to multiplicity in $H_1(M)$.
	\end{enumerate}
	
	Then $\Vol(M) \ge 2V_8.$
\end{lemma}

\begin{proof}
Let $G_i$ be the guts of layer $\sL_i(F)$ and $G_0$ be the $G$ in Lemma \ref{lem:atleast3}. Furthermore, if $G$ and $G'$ are not both 4-STs we assume $G$ is a $T\times I$. By Lemma \ref{lem:intersects}, without loss of generality, we assume $G_0$ intersects $G_1$. We take another $G_i$ with $i>1$. % which will be chosen later such that one of $G_1$ and $G_i$ is $G'$ in Lemma \ref{lem:atleast3}.

We exchange $G_i$ with $G_{i+1},G_{i+2},\cdots$ until it intersects a $G_j$ for some $j$. We still call the new facet surface $F(z)$ and consider the new library bundle structure related to $F(z)$. We consider different cases as follows.
\begin{enumerate}

		\item If $j=0$, $M$ is a library bundle with the library diagram as 
	\begin{equation*}
		\xymatrix{
\cdots  & G_i \ar[r] &G_0 \ar[r] &   G_1  &   G_2  & \cdots  
\save "1,2"."1,4"*[F.]\frm{}
 \restore
 }
\end{equation*}

		we cut $M$ along a union  of horizontal surfaces $S'$ that separate $ G_i \cup G_0 \cup  G_1$ from other guts. By Lemma \ref{lem:gutsoflibrary}(2), for the component of $M \spl S'$ that contains $ G_0 \cup  G_1 \cup G_i$, its pared guts $(N,P)$ is obtained from gluing $G_0$ and a product sutured manifold $W_2$ along some components of the sutured annuli. Furthermore, the pared annular components of $P$ on $R_+$ and $R_-$ come from $G_1$ and $G_i$, respectively.
		\begin{itemize}
			\item If $G_0$ is a $T\times I$, $(N,P)$ is $X\times I$ drilled out an essential curve $c \times \{1/2\}$ with annular cusps on each of $X\times \{0\}$ and $X\times \{1\}$. If $X$ has boundary, then $G$ has at least three annular cusps and a toral cusp. Hence by Theorem \ref{thm:4cusps}, $\Vol(G) \ge 2V_8$. If $X$ is closed, by the atoroidality of $M$, the genus of $X$ is at least 2 and hence $\chi(X) \le -2$. By Theorem \ref{thm:AST}, we still have $\Vol(M) \ge 2 V_8$. 
			\item If $G_0$ is a 4-ST, the underlying manifold $N$ is obtained from an $X\times I$ glued by a thickened annulus $A$. If $\chi(X) = -1$, $X$ is a punctured torus and one component of $A$ is the boundary of $X$ which is homologically trivial in $H_1(M)$. However this means the sutures of $G_0$ are homologically trivial, violating the condition. Hence $\chi(X) \le -2$ and $\Vol(M) \ge 2 V_8$ by Theorem \ref{thm:AST}.
		\end{itemize}

		\item If $j \ge 2$, $M$ is a library bundle with the library diagram as 
	\begin{equation*}
		\xymatrix{
\cdots   &G_0 \ar[r] &   G_1  &  \cdots  & G_i \ar[r] &G_j & \cdots 
\save "1,2"."1,3"*[F.]\frm{}
 \restore
 \save "1,5"."1,6"*[F.]\frm{}
 \restore
 }
\end{equation*}

We cut $M$ along a union of horizontal surfaces $S'$ that separates $ G_0 \cup  G_1$, $ G_i \cup  G_j$ and other guts. Then by Lemma \ref{lem:gutsoflibrary}(1), for the component of $M \spl S'$ containing $ G_0 \cup  G_1$, its pared guts has negative Euler characteristic and hence by Theorem \ref{thm:AST}, have volume $\ge V_8$. So does the pared guts related to $ G_i \cup  G_j$. Hence $\Vol(M) \ge 2 V_8$.

		\item If $j=1$, this means $G_i$ is exchangeable with $G_0$. $M$ is a library bundle with the library diagram as
\begin{equation*}
		\xymatrix{
\cdots  & G_0 \ar[r] &   G_1   & \cdots  \\
& G_i\ar[ur] & &
\save "1,2"."2,3"*[F.]\frm{}
 \restore
 }
\end{equation*}
 We cut $M$ along a union of horizontal surfaces $S'$ that separate $ G_0 \cup  G_1 \cup G_i$ from other guts. By Lemma \ref{lem:gutsoflibrary}(1), for the component of $M \spl S'$ containing $ G_0 \cup  G_1 \cup G_i$, its pared guts $(N,P)$ is a pared sutured manifold $(N,P,R_\pm,\gamma)$ such that $N = X\times I$ with at least two pared annuli on $R_-$ coming from $R_+(G_i)$ and $R_+(G_0)$ which intersect some components of $R_-(G_1)$ nontrivially under projection to $X$. If $\chi(X) \le -2$, by Theorem \ref{thm:AST} $\Vol(M) \ge 2 V_8$. Hence we assume $\Vol(M) < 2 V_8$ and thus $\chi(X) = -1$. Since there are closed curves intersecting each other nontrivially in $X$, $X$ is a punctured torus. Since the components of $R_+(G_0)$ and $R_+(G_i)$ are disjoint, a component of $R_+(G_0)$ and a component of $R_+(G_i)$ are parallel in $X$. 
 
We consider different cases in the condition of Lemma \ref{lem:atleast3}:
 \begin{itemize}
 	\item  Each suture of $G$ and $G'$ is homologically nontrivial in $H_1(M)$ and each suture of $G$ is not isotopic to any suture of $G'$. We choose a $G_i$ such that one of $G_1$ and $G_i$ is $G'$ in Lemma \ref{lem:atleast3}. Since a component of $R_+(G_0)$ and a component of $R_+(G_i)$ are parallel in $X$, $G_i$ is not $G'$ and hence $G_1$ is $G'$. By Lemma \ref{lem:intersects} we can exchange $G_i$ with other sutured guts component in the opposite direction so that $G_i$ intersects $G_j'$. If $j'=1$, the estimation is as in case (1) and if $j' \ge 2$, the estimation is as in Case (2). So we suppose $j'=0$. Similarly in the previous discussion in case (3), we have pared guts as $X' \times I$. We assume $\chi(X') = -1$ and therefore a component of $s(G_1)$ and a component of $s(G_i)$ are parallel in $X'$. Hence a component of $s(G_0)$ and a component of $s(G_1)$ are isotopic which is a contradiction to the assumption. Therefore $\chi(X') \le -2$ and $\Vol(M) \ge 2 V_8$ by Theorem \ref{thm:AST}.
 	
 	\item $G$ and $G'$ are two $T\times I$'s. Then $G_0$ is a $T\times I$ and $R_+(G_0)$ is an annulus. Let $A$ be the annulus connecting $R_+(G_0)$ and a component of $R_+(G_i)$. We can do a cut-and-paste surgery for $\partial (G_0)\cup \partial (G_i)$ along $A$ to have a torus $T$ that separates $G_0\cup G_i$ from other guts. Because $M$ is atoroidal, $T$ bounds a region $N$ which is a solid torus or a $T\times I$. Since each component of $F(z)$ is non-separating, $N$ must contain $G_0\cup G_i$ which means $N$ is obtained from gluing $G_0\cup G_i$ via each a component of $s(G_0)$ and $s(G_i)$. We choose a $G_i$ such that one of $G_1$ and $G_i$ is $G'$ in Lemma \ref{lem:atleast3}. If $G_i$ is a $T\times I$ or 2-ST, $N$ is not a solid torus or a $T\times I$. So $G_i$ is a 4-ST. Hence $G_1$ is $G'$, which is a $T\times I$.

 	 As in the previous case, we can exchange $G_i$ with other sutured guts component in the opposite direction so that $G_i$ intersects $G_j'$. And we can assume $j'=0$. However, this means $\Vol(M) \ge 2 V_8$ or a component of $s(G_0)$ and a component of $s(G_1)$ are isotopic. If the latter holds, because any component of $s(G_i)$ is isotopic to a curve $C_i$ in a distinct boundary component of $M$, respectively for $i=0,1$, we know that $C_0$ is isotopic to $C_1$ which violates the hyperbolicity of $M$.

	\item $G$ is a $T\times I$ and one cannot divide $\Gamma(z)$ into two sets such that the sutures of all sutured guts components in each set are homologous up to multiplicity in $H_1(M)$. Since there is an annuli connecting $R_+(G_0)$ and $R_+(G_i)$, this means a component of $s(G_0)$ is homologous to a component of $s(G_i)$. Because we can choose arbitrary $G_i$ with $i>1$, so we know that we can divide $\Gamma(z)$ into two sets: $\{G_1\}$ and $\{G_i|i\ne 1\}$ which violates the condition.
 \end{itemize}

	\end{enumerate}

\end{proof}

From the previous discussion, we have a useful lemma:

\begin{lemma} \label{lem:nonhomologous}
	Let $M$ be an orientable hyperbolic 3-manifold of finite volume with 3 cusps such that each class in $H_2(M,\partial M)$ is libroid. Suppose $z$ is a vertex of the Thurston sphere. If $\Vol(M) < 2 V_8$, then there exists two sutured guts components $G_0$ and $G_1$ of $\Gamma(z)$ such that any components of $s(G_0)$ and $s(G_1)$ are homologically nontrivial and a component of $s(G_0)$ is not homologous to any multiple of a component of $s(G_1)$.
\end{lemma}
\begin{proof}

	If there is no class in $H_2(M,\partial M)$ that vanishes on two boundary components, from Lemma \ref{lem:suturenontrivial}, we know that the statement holds. So we assume there is a class $y$ in $H_2(M,\partial M)$ that vanishes on two boundary components. Then $\Gamma(y)$ contains two $T\times I$'s, and by Lemmas \ref{lem:2TI4ST} and \ref{lem:atleast3}, we know that $\Vol(M) \ge 2 V_8$, a contradiction.
	\end{proof}

\subsection{Two solid tori with 4 longitudinal sutures in $\Gamma(z)$}

The only case that we have not considered is when $\Gamma(z)$ consists of two 4-STs. For this case, we also have an estimation.

\begin{lemma} \label{lem:24ST}
Let $M$ be an orientable hyperbolic 3-manifold of finite volume with 3 cusps such that each class in $H_2(M,\partial M)$ is libroid. If $z$ is a vertex element of the Thurston sphere such that $\Gamma(z)$ consists of two 4-STs, then $\Vol(M) \ge 1.5V_8.$ 
\end{lemma}
\begin{proof}
	
We use the notation for library bundles as in Section \ref{sec:Library Bundles}. Let $z$ be a vertex element of the Thurston sphere such that $\Gamma(z)$ consists of two 4-STs $G_0$ and $G_1$.

By Lemma \ref{lem:intersects}, without loss of generality, we can assume $G_0$ intersects $G_1$ on $\Sigma_1$. Then if we cut $M$ along $\Sigma_0$, $M \spl \Sigma_0$ will be an $X\times I$ glued by a thickened annulus on each side. By Lemma \ref{lem:gutsoflibrary}(1), the pared guts of $M \spl \Sigma_0$ is $Y\times I$ with a pared annulus on each of $Y\times \{0\}$ and  $Y\times \{1\}$ where $Y$ is the union of negative Euler characteristic components of the minimal essential subsurface containing $R_+(G_0)$ and $R_-(G_1)$ in $\Sigma_1$. If $\chi(Y) \le -2$, by Theorem \ref{thm:AST} we have $\Vol(M) \ge 2V_8$. Therefore we have the following lemma:

\begin{lemma} \label{lem:intersectonce}

If the minimal essential subsurface containing $R_+(G_0)$ and $R_-(G_1)$ in $\Sigma_1$ has Euler characteristic $\le -2$, then $\Vol(M) \ge 2V_8$. 
\end{lemma}

Note that since two components of $R_+(G_0)$ are not isotopic to each other, if two components of $R_+(G_0)$ intersect $R_-(G_1)$, $\chi(Y) \le -2$.

Let $\tilde M$ be the 2-fold covering of $M$ with respect to $\Sigma_0$. Let $\tilde \Sigma_0,\tilde \Sigma_1,\tilde \Sigma_2,\tilde \Sigma_3$ be the liftings of $\Sigma_0, \Sigma_1$ in sequence and $\tilde G_0,\cdots, \tilde G_3$ be the liftings of $G_0,G_1$. So $\tilde \Sigma_0,\tilde \Sigma_2$ are projected to $\Sigma_0$ and $\tilde \Sigma_1,\tilde \Sigma_3$ are projected to $\Sigma_1$. Let $\tilde F(z)$ be the union of liftings of $F(z)$. Then $\tilde F(z) = \bigcup \tilde \Sigma_i$ is a facet surface in $\tilde M$. $\tilde M \spl \tilde F(z)$ is a book of $I$-bundles, i.e., $\tilde M$ is a library bundle with $\tilde \Sigma_i$ as horizontal surfaces and the union of $\tilde G_i$ as spine. Furthermore, $\tilde G_{2i}$ intersects $\tilde G_{2i+1}$ for $i=0,1$ because $G_0$ intersects $G_1$. See Figure \ref{fig:guts4ST}.

\begin{figure}[hbt]
\includegraphics[height = 1.2 in]{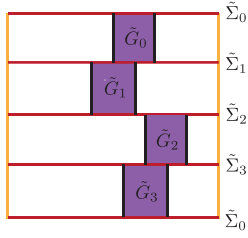}
\caption{The 2-fold covering of $M$ with respect to $\Sigma_0$} \label{fig:guts4ST}
\end{figure}

Before estimating the volume, we want to show some useful facts.

\begin{lemma} \label{lem:isotopicincovering}

	Let $\tilde A$ be an embedded annulus which comes from gluing a union of product annuli in $\tilde M\spl \tilde F(z)$. If a component $\alpha$ of $s(\tilde G_i)$ is connected by $\tilde A$ to a component $\beta$ of $s(\tilde G_{i+2})$ for an $i$, then each of $\alpha$ and $\beta$ can be isotoped in $\tilde F(z)$ to a boundary component of $\tilde F(z)$.	 
	 
\end{lemma}
\begin{proof}
Without loss of generality, suppose $\tilde A$ connects $s(\tilde G_2)$ and $s(\tilde G_0)$. We think of $\tilde A$ as an annulus as well as a union of product annuli. Let $p:\tilde M \to M$ be the covering map. Then $A \stackrel{\Delta}{=} p(\tilde A)$ has both boundary components in $s(G_0)$. %Since the self-intersection points of $A$ have degree 2, they are closed curves. 
We isotope $A$ so that the number of the components of the self-intersection is minimal and do a double curve sum for $A$ to have an embedded surface $A'$. Since the double curve sum of two product annuli is still a union of product annuli, $A'$ is a union of product annuli. Because the double curve sum does not change the Euler characteristic, $\chi(A')= 0 $ and $A' $ consist of annuli and tori. Since the boundary of $ A$ is a union of two components of $s(G_0)$, there is an annulus component $A_0$ of $A'$ such that $\partial A_0 = \partial A$ and $A_0$ is a union of product annuli.

Since $G_0$ is a 4-ST, we can glue $A_0$ with an annulus $A'$ in $\partial G_0$ to have an embedded torus $T_0$. Because the interior of $M$ is hyperbolic, $T_0$ either bounds a solid torus or a $T \times I$. Supposing the first case, a horizontal surface intersects the solid torus nontrivially and hence $A'$ is compressible by a disk, which is a contradiction. Therefore $T_0$ bounds a $T \times I$ and, by considering its intersection with $F(z)$, each boundary component of $A_0$ can be isotoped in $F(z)$ to a boundary component of $\tilde F(z)$. By lifting the isotopy to $\tilde M$, the lemma is proved.

\end{proof}

\begin{lemma}\label{lem:homologousincovering}
	A component of $s(\tilde G_i)$ is neither homologically trivial nor homologous to a multiple of a component of $s(\tilde G_{i\pm 1})$ in $H_1(\tilde M)$ for any $i$.
\end{lemma}
\begin{proof}
	First we suppose that a component of $s(\tilde G_i)$ is homologically trivial in $H_1(\tilde M)$. Then there is an immersed surface $\tilde H$ whose boundary is a component of $s(\tilde G_0)$. By taking the projection of $\tilde H$ to $M$, we have an immersed surface $H$ whose boundary is a component of $s(G_0)$, i.e., this component is homologically trivial, a contradiction to Lemma \ref{lem:nonhomologous}.
	
	Next we suppose that a component of $s(\tilde G_i)$ is homologous to a multiple of a component of $s(\tilde G_{i\pm 1})$ in $H_1(\tilde M)$. Then there is an immersed surface $\tilde H$ connecting $\tilde G_0$ and $\tilde G_1$ via a component of $s(\tilde G_0)$ and a multiple of a component of $s(\tilde G_1)$. By taking the projection of $\tilde H$ to $M$, we have an immersed surface $H$ connecting $G_0$ and $G_1$ which means a component of $s(G_0)$ is homologous to a multiple of a component of $s(G_1)$, a contradiction to Lemma \ref{lem:nonhomologous}.
	
\end{proof}

\begin{lemma} \label{lem:4-puncturedsphere}
	Let $X$ be a 4-punctured sphere and $\alpha, \beta$ be two closed curves. If $\alpha$ intersects $\beta$ nontrivially, each of $\alpha$ and $\beta$ separates two punctures from the others two.% and $\alpha\cup \beta$ divides $X$ into 4 pieces with a puncture.
	\end{lemma}
	\begin{proof}
	If a closed curve separates a puncture from the others, it is parallel to a boundary component of $X$. %Closed curves in a punctured sphere are determined up to isotopy by how they separate the punctures. 
	\end{proof}

Now we continue the proof of Lemma \ref{lem:24ST}. Assume $\Vol(M) < 2V_8$. We define a procedure, called Procedure (E), as follows: 
\begin{enumerate}
	\item exchange $\tilde G_3$ with $\tilde G_0$,
	\item exchange $\tilde G_2$ with $\tilde G_0$,
	\item exchange $\tilde G_3$ with $\tilde G_1$,
	\item exchange $\tilde G_2$ with $\tilde G_1$.
\end{enumerate}

We claim that Procedure (E) cannot repeat for infinitely many times. Otherwise we suppose this procedure can repeat for infinitely many times. We cut $\tilde M$ along $\tilde \Sigma_0 \cup \tilde \Sigma_2$ so that $\tilde M \spl (\tilde \Sigma_0 \cup \tilde \Sigma_2)$ has two components $N_0$ containing $\tilde G_0,\tilde G_1$ and $N_1$ containing $\tilde G_2,\tilde G_3$. Let $\Gamma_0$ be the sutured guts of $N_0$ and $\Gamma_1$ be the sutured guts of $N_1$. Then this procedure will obtain a new facet surface $\tilde F'(z)$ and actually exchange $\Gamma_0$ and $\Gamma_1$ by replacing $\tilde \Sigma_0 \cup \tilde \Sigma_2$ with the new $\tilde \Sigma_0 \cup \tilde \Sigma_2$ in the new $\tilde F'(z)$. So the assumption means that we can exchange $\Gamma_0$ and $\Gamma_1$ for infinitely many times which is a contradiction to Lemma \ref{lem:intersects}.

Therefore we can repeat Procedure (E) finitely many times until one step cannot be done. This is actually equivalent to exchanging $\Gamma_0$ and $\Gamma_1$ as many times as possible and then do first several steps in Procedure (E). 

We still call the new facet surface $\tilde F(z)$ and inherit the notations. 

Note that after exchanging guts, a product annulus in the new $\tilde M \spl \tilde F(z)$ is isotopic to a union of product annuli in the previous $\tilde M \spl \tilde F(z)$ before exchanging guts. Hence Lemma \ref{lem:isotopicincovering} still holds for the new facet surface $\tilde F(z)$. Note that by Lemmas \ref{lem:isotopicincovering} and \ref{lem:homologousincovering}, $s(\tilde G_i)$ and $s(\tilde G_j)$ cannot be connected by nonperipheral product annuli. Hence the assumption of Lemma \ref{lem:gutsoflibrary} holds for any case.

Since at some step we cannot further exchange guts components in Procedure (E), we have 4 possible cases: 
\begin{enumerate}
	\item It stops at the first step of Procedure (E), i.e., $\tilde G_3$ intersects $\tilde G_0$. $\tilde M$ is a library bundle with the library diagram as 
	\begin{equation*}
		\xymatrix{
\tilde G_2 \ar[r] & \tilde G_3 \ar[r] &\tilde G_0 \ar[r] &  \tilde G_1    }
	\end{equation*}

	We state a lemma which can help us to exclude some cases and its proof is at the end of this section.
	
	\begin{lemma} \label{lem:minimalessentail}
		Let $\tilde \Sigma_0$ be the horizontal surface from $\tilde G_3$ to $\tilde G_0$. If the minimal essential subsurface containing $R_+(\tilde G_3)$ and $R_-(\tilde G_0)$ in $\tilde \Sigma_0$ has Euler characteristic $\le -2$, the volume of $\tilde M$ is at least $3 V_8$.
	\end{lemma}
	
	 We cut $\tilde M$ along a horizontal surface $S'$ from $\tilde G_1$ to $\tilde G_2$ in this library bundle structure. From the library diagram we know that for each neighboring layers, their guts intersect each other non-trivially. By Lemma \ref{lem:gutsoflibrary}(3), the pared guts $(N,P)$ of $\tilde M \spl S'$ is a product manifolds $X\times I$ glued by a thickened annulus on each side via $R_-(\tilde G_0)$ and $R_+(\tilde G_3)$ respectively and with annular cusps on each side coming from $\tilde G_2$ and $\tilde G_1$. 
	
	Let $\alpha_1,\alpha_2$ be the components of $R_+(\tilde G_3)$ and $\beta_1,\beta_2$ be the components of $R_-(\tilde G_0)$. 
	
		Denote $Y$ as the union of negative Euler characteristic components of the minimal essential subsurface containing $R_+(\tilde G_3)$ and $R_-(\tilde G_0)$ in $X$.

	Suppose $\Vol(\tilde M) < 3 V_8 $.  By Lemma \ref{lem:minimalessentail}, $\chi(Y) \ge -1$. Because $R_+(\tilde G_3)$ intersects $R_-(\tilde G_0)$, $\chi(Y) =-1$ and $Y$ is a punctured torus. Without loss of generality, we assume that $\alpha_1$ intersects $\beta_1$. Then $\alpha_1$ and $\beta_1$ are in $Y$. 
	
	Because we glue a thickened annuli to $X$ via $\alpha_1$ and $\alpha_2$, each boundary component of $\alpha_1$ is parallel to a boundary component of $\alpha_2$ in the boundary of $N$. Since $\tilde G_2$ intersects $\tilde G_3$, at least a boundary component of $\alpha_2$ meet $R_+(\tilde G_2)$. So $\alpha_2$ is not contained in some components that are thrown away in Lemma \ref{lem:gutsoflibrary}(3). Since $\alpha_2$ is homologically nontrivial, non-parallel to and disjoint with $\alpha_1$, we know that $\alpha_2$ is disjoint with $Y$ and so is $\beta_2$.
	
	If $\chi(X) \le -3$, the volume of $\tilde M$ is at least $3V_8$ by Theorem \ref{thm:AST}. So we assume $\chi(X) \ge -2$. Because $\Gamma(z)$ does not contain $T\times I$'s, we know that $\Sigma_i$, and hence $X$, has non-empty boundary.

	We first consider when $X$ is connected. Since $\chi(X) \ge -2$, the genus of $X$ is 0 or 1. but $X$ contains $Y$ which has genus 1, so $X$ cannot have genus 0 and hence its genus is 1. Because $X$ has non-empty boundary, $X$ is a 2-punctured torus. Then $X\backslash Y$ is a pair of pants which contains $\alpha_2$ and $\beta_2$. However, by Lemma \ref{lem:homologousincovering}, $\alpha_2$ and $\beta_2$ are not homologous to multiples of each other and not homologically trivial which is impossible.

	Next we suppose that $X$ is disconnected. Then $X$ is a union of 2 surfaces $Y$ and $Y'$ each with Euler characteristic -1. Since $\beta_2$ is disjoint with $Y$, $\beta_2$ is in $Y'$.

In the construction of $(N,P)$, we glue a thickened annulus to each of $X\times \{0\}$ and $X\times \{1\}$ respectively. We have a sutured manifold whose $R_+$ part, denoted as $X'$, is obtained from $X$ by cutting along $\beta_1,\beta_2$ and gluing back two annuli $A_1,A_2$ where $A_1\cup A_2$ are actually $R_+(G_0)$. %TODO(figure)
	
	If $Y'$ is a punctured torus, since $\beta_2$ is not homologically trivial in $M$, $\beta_2$ is a non-separating curve in $Y'$. Then $X'$ is a 2-punctured torus and the cores of $A_1,A_2$ are two non-separating curves having the same slope. Hence if we let $A'$ be a component of $R_-(\tilde G_1)$ in $X'$ intersecting one of $A_1,A_2$, $A'$ must intersect both of them. Since $A_1,A_2$ and $A'$ lie in the guts $\Gamma_1$, by projecting them onto $M$, the projection of $X'$ is the union of negative Euler characteristic components of minimal essential subsurface containing $R_+(G_0)$ and $R_-(G_1)$. So by Lemma \ref{lem:intersectonce}, the volume of $M$ is at least $2V_8$.
	
	If $Y'$ is a 3-punctured sphere, $\beta_2$ is parallel to a boundary of $Y'$ and $X'$ is a 4-punctured sphere. Since $A_1$ or $A_2$ intersects a component of $R_-(\tilde G_1)$, by Lemma \ref{lem:4-puncturedsphere}, the minimal essential subsurface which contains $R_+(\tilde G_0)$ and $R_-(\tilde G_1)$ is the 4-punctured sphere $X'$. Similarly by projecting them onto $M$, the projection of $X'$ is the union of negative Euler characteristic components of minimal essential subsurface containing $R_+(G_0)$ and $R_-(G_1)$. So by Lemma \ref{lem:intersectonce}, the volume of $M$ is at least $2V_8$.

\item It stops at the second step of Procedure (E), i.e., $\tilde G_2$ intersects $\tilde G_0$. $\tilde M$ is a library bundle with the library diagram as 
	\begin{equation*}
		\xymatrix{
\tilde G_2 \ar[r]  \ar[dr] &\tilde G_0 \ar[r] &  \tilde G_1 \\
& \tilde G_3&  }
	\end{equation*} 

	 We cut $\tilde M$ along a horizontal surface $S'$ from $\tilde G_1$ to $\tilde G_2$. If $\tilde G_3$ is exchangeable with $\tilde G_1$, we move $\tilde G_3$ to the third layer. Otherwise, we keep it in the same place. 
	 Let $X$ be as in Lemma \ref{lem:gutsoflibrary}(2). Then the pared guts of $\tilde M \spl S'$ is a pared sutured manifold $(N,P, R_\pm,\gamma)$ whose $R_-$ is $X$. Let $\alpha_1,\alpha_2$ be the the cores of $R_-(\tilde G_0)$. Since the projections of $\alpha_1,\alpha_2$ onto $R_+(N)$ are parallel and $\tilde G_0$ intersects $\tilde G_1$, $X$ contains $\alpha_1,\alpha_2$. Since $\tilde G_2$ intersects $\tilde G_3$, $X$ contains at least a component of $R_-(\tilde G_3)$ and $R_+(\tilde G_2)$ respectively, denoted as $\beta$ and $\delta$, such that $\beta$ intersects $\delta$. By Lemma \ref{lem:homologousincovering}, $\alpha_1,\alpha_2, \beta$ are three disjoint curves in $X$ and they are not parallel to each other. Moreover, one of $\alpha_1,\alpha_2$, say $\alpha_1$, intersects a curve coming from $R_+(\tilde G_2)$. 
	
	We first consider when $X$ is connected. Because $\alpha_1$ and $\beta$ are homologically nontrivial and homologically nonparallel, the union of them is a pair of curves that decomposes $X$ into two pairs of pants. Then $\alpha_2$ is a boundary of $X$ and $\alpha_1,\alpha_2,\beta$ bound a pair of pants. Because $\tilde G_0$ is a 4-ST, $\alpha_1$ is homologous to $\alpha_2$ in $\tilde M$ so that $\beta$ is homologous to a multiple of $\alpha_1$ in $\tilde M$, which is a contradiction to Lemma \ref{lem:homologousincovering}.
	
	Suppose $X$ is disconnected. Then $X$ is a union of 2 surfaces $X_1$ and $X_2$ with Euler characteristic -1. Because both $\alpha_1$ and $\beta$ intersect a component of $R_-(\tilde G_2)$ respectively and they are homologically nontrivial and nonparallel, $X_1$ and $X_2$ are punctured tori and $\alpha_1$ and $\beta$ are in $X_1$ and $X_2$ respectively. Then $\alpha_2$ is the boundary of $X_1$ or $X_2$ which is a contradiction to Lemma \ref{lem:homologousincovering}.

%	So we have $\chi(X) \le -3$ and $\Vol(M) \ge 1.5V_8$ by Theorem \ref{thm:AST}.

\item It stops at the third step of Procedure (E), i.e., $\tilde G_3$ intersects $\tilde G_1$. $\tilde M$ is a library bundle with the library diagram as 
	\begin{equation*}
		\xymatrix{
\tilde G_2 \ar[dr] &\tilde G_0 \ar[r] &  \tilde G_1 \\
& \tilde G_3 \ar[ur] &  }
\end{equation*}

We cut $M$ along a horizontal surface $S'$ from $\tilde G_1$ from $\tilde G_2$. Then we have the same estimation as in Case (2) if we think of $\tilde G_0,\tilde G_1,\tilde G_2,\tilde G_3$ here as $\tilde G_3,\tilde G_2,\tilde G_1,\tilde G_0$ in Case (2) respectively.

\item It stops at the fourth step of Procedure (E), i.e., $\tilde G_2$ intersects $\tilde G_1$. $\tilde M$ is a library bundle with the library diagram as 
	\begin{equation*}
		\xymatrix{
 \tilde G_0 \ar[r] &  \tilde G_1 \\
\tilde G_2 \ar[r]\ar[ur] & \tilde G_3 
 }
\end{equation*}

We pick a union $F'$ of two $\tilde \Sigma_i$ which separate $\tilde G_0 \cup \tilde G_2$ and $\tilde G_1 \cup \tilde G_3$. Then $\tilde M$ is a library bundle with respect to $F'$ and its spine is $\tilde G_0 \cup \tilde G_2, \tilde G_1 \cup \tilde G_3$. In this new library bundle structure, we cut $M$ along a horizontal surface $S'$ from $\tilde G_1 \cup \tilde G_3$ from $\tilde G_0 \cup \tilde G_2$. By Lemma \ref{lem:gutsoflibrary}(1), the pared guts of $\tilde M \spl S'$ contains a product manifolds $X\times I$ with annular cusps on $X \times \{0\}$ from $R_+(\tilde G_0),R_+(\tilde G_2)$ and annular cusps on $X \times \{1\}$ coming from $R_-(\tilde G_1),R_-(\tilde G_3)$. 
	
	Because $G_0$ intersects $G_1$, we know that $\tilde G_3$ intersects $\tilde G_2$ and $\tilde G_0$ intersects $\tilde G_1$. Since we already exchange $\tilde G_3$ with $\tilde G_0$, $R_+(\tilde G_0)$ is disjoint with $R_-(\tilde G_3)$. So there is a curve $c_3$ in $X$ which intersects $c_2$ and disjoint with $c_0$ and $c_1$ which intersect each other, where $c_0,c_1,c_2,c_3$ is the core of a component of $R_+(\tilde G_0),R_-(\tilde G_1),R_+(\tilde G_2),R_-(\tilde G_3)$ respectively. Because of Lemma \ref{lem:isotopicincovering}, $c_1,c_3$ are not parallel in $X$ and neither are $c_0,c_2$. 
	
	Because $\Vol(M) < 1.5 V_8$, by Theorem \ref{thm:AST}, $\chi(X) \ge -2$.
	
	Because we exchange $\Gamma_0$ and $\Gamma_1$ as many times as possible and then do the first 3 steps in Procedure (E), we can actually reverse the last 3 operations and it does not affect the intersection pattern for $R_+(\tilde G_0)$ and $R_-(\tilde G_1)$. Let $Y$ be the minimal essential subsurface containing $c_0$ and $c_1$. Since $R_+(\tilde G_0)$ and $R_-(\tilde G_1)$ lie in the guts $\Gamma_0$, by projecting them onto $M$, the projection of $Y$ is a subsurface of the minimal essential subsurface containing $R_+(G_0)$ and $R_-(G_1)$.  
	
	 Because $\Vol(M) < 1.5 V_8$, by Lemma \ref{lem:intersectonce}, the Euler characteristic of $Y$ is $-1$ and hence $Y$ is a punctured torus. This means the algebraic intersection number $\langle c_0,c_1\rangle \ne 0$ and similarly $\langle c_2,c_3\rangle \ne 0$. Hence $c_0,c_1,c_2,c_3$ are homologically nontrivial in $H_1(X,\partial X)$.
	
	First we suppose $X$ is connected. Then $X$ is an at most 2-punctured torus. Because $c_1$ and $c_3$ are disjoint and homologically nontrivial in $H_1(X,\partial X)$, $X$ is a 2-punctured torus and they separate two punctures from each other. Hence $[c_1] = \pm [c_3]$ in $H_1(X,\partial X)$ and therefore $\langle c_0,c_3\rangle = \langle c_0,c_1\rangle \ne 0$ which is a contradiction to the fact that $c_0$ is disjoint with $c_3$.
	
	If $X$ is disconnected, because $c_0$ is disjoint with $c_3$, $X$ is a union of punctured tori $X_1$ and $X_2$ with $c_0$ and $c_1$ in $X_1$ and $c_2$ and $c_3$ in $X_2$. Because $\tilde G_2$ intersects $R_-(\tilde G_1)$, there is a component $c'_2$ of $R_+(\tilde G_2)$ intersecting a component $c'_1$ of $R_-(\tilde G_1)$ in either $X_1$ or $X_2$. However by Lemma \ref{lem:isotopicincovering} and Lemma \ref{lem:homologousincovering}, $c'_2$ is neither isotopic to $c_0$ nor homologically nontrivial. Hence $c'_2$ is not in $X_1$. Similarly $c'_1$ is not in $X_2$. So it is impossible for $c'_2$ to intersect $c'_1$.

\end{enumerate}

Therefore we complete the proof of Lemma \ref{lem:24ST}.

\end{proof}

\begin{proof}[Proof of Lemma \ref{lem:minimalessentail}]
	
By Lemma \ref{lem:intersects}, we exchange $\tilde G_2$ with $\tilde G_1,\tilde G_0,\tilde G_3,\ldots$ in the \emph{opposite} direction until it intersects some $\tilde G_i$. 
	\begin{itemize}
		\item If $i=1$, the library diagram is \xymatrix{
\tilde G_3 \ar[r] &\tilde G_0 \ar[r] &  \tilde G_1 \ar[r] & \tilde G_2  }

		We cut $\tilde M$ along a union of two horizontal surfaces $S'$ which separate $\tilde G_1 \cup \tilde G_2$ from $\tilde G_0 \cup \tilde G_3$. Since the minimal essential subsurface containing $R_+(\tilde G_3)$ and $R_-(\tilde G_0)$ in $\tilde \Sigma_0$ has Euler characteristic $\le -2$, by Lemma \ref{lem:gutsoflibrary}(1), the pared guts of related to $\tilde G_3$ and $\tilde G_0$ has at most Euler characteristic $-2$ and hence the volume $\ge 2V_8$. Moreover, the pared guts of $\tilde M \spl S'$ containing $\tilde G_1 \cup \tilde G_2$ has volume $\ge V_8$. So $\Vol(\tilde M) \ge 3V_8$. 
		\item If $i= 0$, the library diagram is 
\xymatrix{
\tilde G_3 \ar[r] &\tilde G_0 \ar[r]\ar[dr] &  \tilde G_1  \\
     &&  \tilde G_2  }

		We cut $\tilde M$ along a horizontal surface $S'$ from $\tilde G_1 \cup \tilde G_2$ to $\tilde G_3$. Then we have a library sutured manifold $\tilde M \spl S'$ with 3 layers containing $\tilde G_3$, $\tilde G_0$ and $\tilde G_1 \cup \tilde G_2$. By Lemma \ref{lem:gutsoflibrary}(2), the pared sutured guts $(N,P)$ of $\tilde M \spl S'$ is obtained from a product manifolds $X\times I$ by gluing a thickened annulus via two curves $\alpha_1,\alpha_2$ on $X\times \{1\}$ coming from $\tilde G_0$ and with annular cusps on each side. By the assumption of the lemma, $\chi(X) \le -2$. Let $X'$ be the $R_+$ side of $N$. Then $\chi(X') = \chi(X) \le -2$.
		
	In the following, we suppose $\Vol(M) < 1.5 V_8$, i.e., $\Vol(\tilde M) < 3 V_8$. By Theorem \ref{thm:AST}, we know that $\chi(X') = -2$.
	
	\begin{enumerate}
		
	\item We first consider when $X'$ is connected. %If $X'$ has genus at least 2, since $\Gamma(z)$ does not contain the boundary of $M$, $F(z)$ and hence $X$ has at least one boundary component. So $\chi(X') \le 2 - 2\times 2 -1 = -3$. By Theorem \ref{thm:AST}, we have $\Vol(\tilde M) \ge 3V_8$ which means $\Vol(M) \ge 1.5V_8$. Therefore we assume $X'$ is a surface of genus 0 or 1 and $\chi(X') = -2$. 

	If $X'$ is genus 1, then $X'$ has 2 boundary components, i.e., $X'$ is a 2-punctured torus. By Lemma \ref{lem:intersectonce}, we can assume the minimal essential subsurface containing $R_+(\tilde G_0)$ and $R_-(\tilde G_1)$ has Euler characteristic -1. Hence there is a component $\beta$ of $R_-(\tilde G_1)$ intersecting exactly a component $\alpha_1$ of $R_+(\tilde G_0)$ in $X'$ and the minimal essential subsurface containing $\beta$ and $\alpha_1$ is a punctured torus. Let $\alpha_2$ be the other component of $R_+(\tilde G_0)$ and $\delta$ be a component of $R_-(\tilde G_2)$ that intersects $R_+(\tilde G_0)$.

	Since $\alpha_2$ is disjoint with $\alpha_1,\beta$ and not homologically trivial, $\alpha_2$ is parallel to a boundary component of $X'$. So $\delta$ intersects $\alpha_1$. By the fact that $\delta$ is disjoint with $\beta$ and Lemma \ref{lem:homologousincovering}, $\beta$ and $\delta$ separate the two punctures in $X'$. Hence $\delta,\beta$ and $\alpha_2$ bound a pair of pants. By projecting them onto $M$, we have a component of $s(G_1)$ and 2 copies (without orientation) of a component of $s(G_0)$ whose union bounds an immersed surface. Then either a component of $s(G_0)$ is homologically trivial or $s(G_1)$ is homologous to a multiple of $s(G_0)$, which is a contradiction to Lemma \ref{lem:nonhomologous}.

	If $X'$ is genus 0, $X'$ has 4 boundary components, i.e., $X'$ is a 4-punctured sphere. Let $\alpha$ be a component of $R_+(\tilde G_0)$ that intersects a component $\beta$ of $R_-(\tilde G_1)$. By Lemma \ref{lem:4-puncturedsphere}, each of $\alpha_1$ and $\beta$ is a closed curve that separates 2 punctures from the others. Similarly, a component $\delta$ of $R_-(\tilde G_2)$ intersects $\alpha$ and hence $\delta$ also separates 2 punctures from the other. Since $\beta$ and $\delta$ are disjoint, they are parallel, which is a contradiction to Lemma \ref{lem:homologousincovering}.
		
	\item	If $X'$ is disconnected and $\chi(X') = -2$, then $X'$ is a union of 2 surfaces $X'_1$ and $X'_2$ with Euler characteristic -1. Because a component $\alpha_1$ of $R_+(\tilde G_0)$ intersects a component $\beta$ of $R_-(\tilde G_1)$, we let $X'_1$ be the punctured torus containing $\alpha_1$ and $\beta$. Let $\delta$ be a component of $R_-(\tilde G_2)$ that intersects $R_+(\tilde G_0)$. By Lemma \ref{lem:homologousincovering}, $\delta$ and $\beta$ are not homologous. So $\delta$ intersects another component $\alpha_2$ of $R_+(\tilde G_0)$ in $X'_2$ which is also a punctured torus. 
	
	We then repeat exchanging $\tilde G_3$ with $\tilde G_1 \cup \tilde G_2$, $\tilde G_0$ in the \emph{opposite} direction until $\tilde G_3$ intersects a $\tilde G_i$. Then we have one of the following library diagrams: 
		\begin{equation*}
		\xymatrix{
\tilde G_0 \ar[r]\ar[dr] &  \tilde G_1\ar[r] &  \tilde G_3 & \tilde G_0 \ar[r]\ar[dr] &  \tilde G_2\ar[r] &  \tilde G_3 & \tilde G_0 \ar[r]\ar[dr]\ar[ddr] &  \tilde G_1\\
     &  \tilde G_2&  & &  \tilde G_1& & &  \tilde G_2 \\
      & &&&&&&\tilde G_3 }
	\end{equation*}
	
	We cut along a horizontal surface $S'$ from $\tilde G_3$ to $\tilde G_0$ and take the pared guts of the complement.

	We first consider the first diagram. 
	
	If $\tilde G_2$ also intersects $\tilde G_3$, by Lemma \ref{lem:gutsoflibrary}(2), the pared sutured guts $(N,P)$ of $\tilde M \spl S'$ is obtained from a product manifold $X''\times I$ by gluing two thickened annuli on $X''\times \{1\}$ coming from $\tilde G_1$ and $\tilde G_2$ and with annular cusps on each side. Furthermore, $X''$ contains $X'$ and components of $R_-(\tilde G_1)$ and $R_-(\tilde G_2)$. 	
	
	If $\tilde G_2$ does not intersect $\tilde G_3$, we move $\tilde G_2$ to the third level in the library diagram as follows.
	
		\begin{equation*}
		\xymatrix{
\tilde G_0 \ar[r] \ar@{..>}[drr] &  \tilde G_1\ar[r] &  \tilde G_3 \\
     & & \tilde G_2 }
	\end{equation*}
	Then this library structure satisfies the condition of Lemma \ref{lem:gutsoflibrary}(2). Therefore, the pared sutured guts $(N,P)$ of $\tilde M \spl S'$ is obtained from a product manifolds $X''\times I$ by gluing a thickened annulus on $X''\times \{1\}$ coming from $\tilde G_1$ and with annular cusps on each side. Furthermore, $X''$ contains $X'$ and components of $R_-(\tilde G_1)$. 
	
	Hence in both cases, since the component of $R_-(\tilde G_1)$ other than $\beta$ is homologically nontrivial, it cannot be in $X'$. So $\chi(X'') \le \chi(X')-1 = -3$ and by Theorem \ref{thm:AST}, $\Vol(\tilde M)\ge 3V_8$. 
	
	Similarly, the same estimation works for the second diagram.
	
	As for the third diagram, by Lemma \ref{lem:gutsoflibrary}(1), the pared sutured guts $(N,P)$ of $\tilde M \spl S'$ is a product manifolds $X''\times I$ with annular cusps on each side. Furthermore, $X''$ contains $X'$, $\alpha_1$, $\alpha_2$, $\beta$ and $\delta$. Let $\zeta$ be a component of $R_-(\tilde G_3)$ that intersects $\alpha_1$ or $\alpha_2$. By Lemma \ref{lem:isotopicincovering} and \ref{lem:homologousincovering}, $\zeta$ is homologically trivial and not parallel to $\beta$ or $\delta$. Hence it cannot be in $X'$. So $\chi(X'') \le \chi(X')-1 = -3$ and again by Theorem \ref{thm:AST}, $\Vol(\tilde M)\ge 3V_8$.

	\end{enumerate}
	
	\item If $i= 3$, the library diagram is 
\xymatrix{
\tilde G_3 \ar[r]\ar[dr] &\tilde G_0 \ar[r] &  \tilde G_1  \\
     & \tilde G_2 &  }

	The proof of this case is almost the same as Case (2) in the proof of Lemma \ref{lem:24ST}. Only need two things need to be changed:
	\begin{enumerate}
		\item Switch the notation of $\tilde G_2$ and $\tilde G_3$ so $\beta$ is a component of $R_-(\tilde G_2)$ and $\delta$ is a component of $R_+(\tilde G_3)$
		\item When $X$ is connected, $\alpha_1,\alpha_2,\beta$ bound a pair of pants. Here we project them to $M$ so that we know either 3 copies (without orientation) of a component of $s(G_0)$ bounds an immersed surface in $M$. Therefore this component of $s(G_0)$ is homologically trivial in $H_1(M)$ which is a contradiction to Lemma \ref{lem:nonhomologous}.
	\end{enumerate} 
	
	\end{itemize}

	\end{proof}

Finally, we can prove Theorem \ref{thm:2TIor4ST}.

\begin{proof}[Proof of Theorem \ref{thm:2TIor4ST}]
If there is no class in $H_2(M,\partial M)$ that vanishes on two boundary components of $M$, then $\Gamma(z)$ contains at least one 4-ST. By Lemma \ref{lem:2TI4ST}(2), Lemma \ref{lem:24ST} and Lemma \ref{lem:atleast3}(1), we know that the theorem holds. If there is a class $y$ in $H_2(M,\partial M)$ that vanishes on two boundary components of $M$, the sutured guts $\Gamma(y)$ contains two $T\times I$'s, hence by Lemma \ref{lem:2TI4ST}(1) and Case (2) of Lemma \ref{lem:atleast3}, the volume of $M$ is at least $2V_8$.
\end{proof}

\section{Volume of Orientable Hyperbolic 3-Manifolds with 3 Cusps as Library Bundles} \label{sec:volume:general}

In this section we are going to prove Theorem \ref{thm:volumeofbook}. Before proving the theorem, we need some preparations.

\begin{lemma} \label{lem:nononseparating}
	Let $M$ be a manifold in Theorem \ref{thm:volumeofbook}. If there is a non-separating closed surface in $M$, then $\Vol(M)\ge 2V_8$.
\end{lemma}
\begin{proof}
	From the condition, we know that there is an element $z$ in $H_2(M,\partial M)$ such that it vanishes on each boundary component of $M$. So we can pick a facet surface $F(z)$ and have the sutured guts $\Gamma(z)$. Since $F(z)$ is disjoint with $\partial M$, $\partial \Gamma(z)$ contains $\partial M$, i.e., the three sutured tori. Because $\Gamma(z)$ only consists of solid tori and $T\times I$'s, $\Gamma(z)$ contains three $T\times I$'s. Therefore, by Lemma \ref{lem:2TI4ST}, $\Vol(M) \ge 2V_8$.
\end{proof}

\begin{definition}
	Let $M$ be a 3-manifold with toral boundary and $P$ be a union of some boundary components of $M$. We denote $D_P(M)$ as the result of doubling $M$ along $P$ and $D_P(X)$ as the result of doubling $X$ along $P$ for any subset $X$ of $M$. For a second homology class $z$ in $H_2(M,\partial M)$, we can have $D_P(z)$ to be the homology class of $D_P(z)$ in $H_2(D_P(M),\partial D_P(M))$ where $S$ represents $z$ in $H_2(M,\partial M)$.
	
	Specifically, when $P=\partial M$, we use $D(M)$, $D(X)$ and $D(z)$ by omitting $P$. 
\end{definition}

\begin{definition}\label{def:doubledguts}
	Let $M$ be a compact orientable 3-manifold with toral boundary, and $z$ be an element in $H_2(M,\partial M)$. Let $\Gamma(z)$ be the sutured guts of $z$. We cut along a maximal collection of non-parallel horizontal surfaces in $D(\Gamma(z))$ and then a maximal product annuli and tori. We call the \emph{doubled guts} of $z$ as the non-product components of the resulting sutured manifold, i.e., the \emph{horizontally prime guts} (see \cite[Definition 3.19]{AZ1}) of $D(\Gamma(z))$.
\end{definition}

\begin{lemma} \label{lem:thurstonnormdouble}
	Let $M$ be a compact orientable 3-manifold with toral boundary, and $P$ be a union of some boundary components of $M$. Let $z$ be an element in $H_2(M,\partial M)$. If $S$ is a norm-minimizing surface representing $z$, then $D_P(S)$ is also norm-minimizing.
\end{lemma}
\begin{proof}
	Suppose $D_P(S)$ is not norm-minimizing. Then we have a surface $U$ representing $D_P((z)$ in $D_P(M)$ such that $\chi_-(U) < \chi_-(D_P(S)) = 2 x(z)$. We let $D_P(M)$ be $M \cup M'$ where $M'$ is a mirrored copy of $M$. Then either $\chi_-(U\cap M)$ or $\chi_-(U\cap M')$ is smaller than $x(z)$. Note that $U\cap M$ and the mirrored copy of $U\cap M'$ in $M$ represent $z$ in $H_2(M,\partial M)$. This contradicts the definition of Thurston norm.
\end{proof}

\begin{lemma}\label{lem:2defdoubledguts}
	Let $M$ be an orientable hyperbolic 3-manifold of finite volume such that every element in $H_2(M,\partial M)$ is a libroid class. Let $z$ be an element in $H_2(M,\partial M)$ and $P$ be a boundary component of $M$. Then the horizontally prime guts of $D_P(\Gamma(z))$ is equivalent to the guts of $D_P(z)$ in $D_P(M)$.
\end{lemma}

\begin{proof}

Note that $P$ is the only essential torus in $D_P(M)$ and it is null-homologous. Hence the Thurston norm of $D_P(M)$ is non-degenerate. Therefore by \cite[Theorem 1.1]{AZ1}, the guts of $D_P(z)$ in $D_P(M)$ is well-defined.

 Let $F(z)$ be a facet surface representing $z$. Then $D_P(F(z))$ will be a collection of norm-minimizing surfaces representing $D_P(z)$ in $D_P(M)$. Since each component of $M \spl F(z)$ contains one component of $\Gamma(z)$ by Lemma \ref{lem:label}, each component of $D_P(M) \spl D_P(F(z))$ contains the double of one component of $\Gamma(z)$. %Hence we can take a facet surface $F(D_P(z))$ representing $D_P(z)$ and containing $D_P(F(z))$. 

Let $G$ be a component of $\Gamma(z)$ and $N$ be the component of $M \spl F(z)$ containing $G$. 

If $G$ is a $T\times I$ with 2 sutures on a boundary component and a sutured toral boundary, $D_P(G)$ is a $T\times I$ with 2 isotopic sutures on each boundary component. We take a taut surface $S'_N$ homologous to the $R_+$ part of $D_P(G)$ such that it decomposes $D_P(G)$ into two 4-ST's. If $G$ is a 4-ST or a 2-ST, we can also have $S'_N$ that decomposes $D_P(G)$ into two 4-ST's or 2-ST's. Then the horizontally prime guts of $D_P(G)$ is the union of these two solid tori. By extending $S'_N$ to the product sutured manifold $D_P(N)\backslash D_P(G)$, we have a taut surface $S_N$ which represents $z$. Furthermore, $\{S_N\}$ is a maximal collection of taut surfaces representing $z$ in $D_P(N)$ and non-isotopic to $R_\pm(N)$.

The union of $D_P(F(z))$ and $S_N$ for each component $N$ of $M \spl F(z)$ is a facet surface $F(D_P(z))$ representing $D_P(z)$ in $D_P(M)$. Since the guts of $F(D_P(z))$ is equivalent to the horizontally prime guts of $D_P(\Gamma(z))$ , we know the lemma is true.

\end{proof}

\begin{lemma}\label{lem:invariantdoubledguts}
	Let $M$ be an orientable hyperbolic 3-manifold of finite volume such that every element in $H_2(M,\partial M)$ is a libroid class. Let $y,z$ be in an open Thurston cone $\Delta$ of $M$ such that $y$ vanishes on exactly one boundary component of $M$. Then there is a $w$ in the open segment $(y,z)$ such that the doubled guts of $y$ is equivalent to the doubled guts of $u$ in $D(M)$ for every element $u$ in $[y,w]$.
	
	\end{lemma}

\begin{proof}
Let $y$ vanish on a boundary component $P$. We take $w$ small enough so that every element in $[w,y]$ does not vanish on any boundary component other than $P$.

%Denote $D_P(u)$ to be the homology class of $D_P(S)$ in $H_2(D_P(M),\partial D_P(M))$ where $S$ represents $u$ in $H_2(M,\partial M)$. 
Note that $P$ is the only essential torus in $D_P(M)$ and it is null-homologous. Hence the Thurston norm of $D_P(M)$ is non-degenerate. Since $y$ and $w$ are in the same open Thurston cone, there is a class $w'$ such that we have $\chi_-(ky) = \chi_-(w')+\chi_-(w)$ and $w' = ky-w$ for some positive number $k$. Therefore $\chi_-(kD(y)) = \chi_-(D(v))+\chi_-(D(w))$. By Lemma \ref{lem:thurstonnormdouble}, we know that $D_P(y)$ and $D_P(w)$ are both in an open Thurston cone of $D_P(M)$. Besides, there is an open segment $(a, b)$ containing $D(y)$ and $D(w)$ in the open Thurston cone such that the restrictions of $a$ and $b$ on each boundary component are not in opposite orientations. By \cite[Theorem 1.2]{AZ1}, the guts of $D_P(y)$ is equivalent to the guts of $D_P(u)$ in $D_P(M)$ for every element $u$ in $[y,w]$.

By Lemma \ref{lem:2defdoubledguts}, we know that the horizontally prime guts of $D_P(\Gamma(y))$ is equivalent to the horizontally prime guts of $D_P(\Gamma(u))$. Let $S_P(y)$ and $S_P(u)$ be the union of a maximal collection of non-parallel horizontal surfaces in $D_P(\Gamma(y))$ and $D_P(\Gamma(z))$, respectively. By gluing themselves along boundary components other than $P$, we have a union of non-parallel horizontal surfaces $S(y)$ and $S(u)$ in $D(\Gamma(y))$ and $D(\Gamma(z))$, respectively. Because $y$ and $u$ do not vanish on any boundary component other than $P$, we know that each $S(y)$ and $S(u)$ is the union of a maximal collection of non-parallel horizontal surfaces. Therefore the horizontally prime guts of $D(\Gamma(y))$ is equivalent to the horizontally prime guts of $D(\Gamma(u))$.

\end{proof}

By Lemma \ref{lem:invariantdoubledguts} and \cite[Theorem 1.2]{AZ1}, the following theorem holds.

\begin{theorem}\label{thm:invariancedoubledguts} 

		Let $M$ be an orientable hyperbolic 3-manifold of finite volume such that every element in $H_2(M,\partial M)$ is a libroid class and $\Delta$ be an open Thurston cone. If no element in $\Delta$ vanishes on two boundary component, then for two classes $y,z$ in $\Delta$, the doubled guts of $y$ is equivalent to the doubled guts of $z$ in $D(M)$.

\end{theorem}

\begin{proof}
	Let $(v,w)$ be an open segment containing $y,z$ in $\Delta$. We divide $(v,w)$ by finitely many points such that each resulting open component does not contain two elements whose restrictions on a boundary component are in opposite orientations. By \cite[Theorem 1.2]{AZ1}, guts are invariant for second homology classes in each open segment. By Lemma \ref{lem:invariantdoubledguts}, we know that each point on $(v,w)$ has a neighborhood in $(v,w)$ such that double guts for second homology classes in this neighborhood are invariant. Combining these, we know that Theorem \ref{thm:invariancedoubledguts} is true.
\end{proof}

Now we start proving Theorem \ref{thm:volumeofbook}.

\begin{proof}[Proof of Theorem \ref{thm:volumeofbook}]
	By the well-orderness the volumes of hyperbolic 3-manifolds (see \cite[Theorem 5.12.1]{thurston1979geometry}), there is an orientable hyperbolic 3-cusped manifold $\Mb$ satisfying the condition in Theorem \ref{thm:volumeofbook} and having the minimal volume among all such 3-manifolds. Suppose $\Vol(\Mb) < 2 V_8$.

 There is a long exact sequence for $(\Mb, \partial \Mb)$:

\[
\cdots \to H_2(\Mb) \to H_2(\Mb,\partial \Mb) \stackrel{\partial}{\to} H_1(\partial \Mb) \to \cdots
\]

By Lemma \ref{lem:nononseparating}, we know that there is no non-separating closed surface in $\Mb$ and hence the map from $H_2(\Mb)$ to $H_2(\Mb, \partial \Mb)$ is 0. Therefore the map $\partial$ is injective. Because the image of $\partial$ is of rank 3 by Lemma \ref{lem:halflives}, $H_2(\Mb,\partial \Mb)$ is 3-dimensional. Since the Betti number of $\Mb$ is 3 and, for each boundary component $P$, $H_1(P) = \Z^2$, there is an element $z$ in $H_2(\Mb,\partial \Mb)$ such that $\partial z = 0$ in the $H_1(P)$ summand of $H_1(\partial \Mb)$, i.e., $z$ vanishes on $P$.

If there is a $z$ in $H_2(\Mb,\partial \Mb)$ vanishing on two boundary components, the sutured guts $\Gamma(z)$ contains 2 $T\times I$'s. By Lemmma \ref{lem:2TI4ST}, $\Vol(\Mb) \ge 2V_8$. So we assume there is no such class.

Let $\partial_1 \Mb,\partial_2 \Mb,\partial_3 \Mb$ be the components of $\partial \Mb$ and denote $y_i$ as a class in $H_2(\Mb,\partial \Mb)$ vanishing on $\partial_i \Mb$ for $1 \le i \le 3$. Then $\Gamma(y_i)$ contains a $T\times I$.

If $y_i$ is a vertex of the Thurston sphere, by Lemma \ref{lem:nonhomologous}, there exists two sutured guts component $G_0$ and $G_1$ of $\Gamma(y_i)$ such that  any components of $s(G_0)$ and $s(G_1)$ are homologically nontrivial and a component of $s(G_0)$ is not homologous to any multiple of a component of $s(G_1)$. Since $\Vol(\Mb) < 2V_8$, by Lemma \ref{lem:atleast3}(3), we can divide $\Gamma(y_i)$ into two sets $\Gamma_0$ and $\Gamma_1$ such that the sutures of all sutured guts components in each set are homologous up to multiplicity. By the property of $G_0$ and $G_1$, we can assume that $\Gamma_0$ contains $G_0$ and $\Gamma_1$ contains $G_1$. By Lemma \ref{lem:atleast3}(2), there is only one $T\times I$ component in $\Gamma(y_i)$. Let the $T\times I$ be in $\Gamma_0$. By Lemma \ref{lem:sequence}, we can take $u$ in an edge of the Thurston sphere with $y_i$ as an endpoint and $v$ in a face meeting the edge such that $\Gamma(y_i)$ has a sequence of 2 sutured decomposition
\[
	\Gamma(y_i) \stackrel{\Sigma_1}{\overline \leadsto} \Gamma(u)  \stackrel{\Sigma_2}{\overline \leadsto} \Gamma(v).
	\]
where $\Sigma_1$ and $\Sigma_2$ are properly norm-minimizing surfaces in $M$.

As in the proof of Lemma \ref{lem:suturenontrivial}, we can choose $e$ such that $\Sigma_1$ intersects a component of $s(\Gamma_1)$. Since the sutures of all sutured guts components of $\Gamma_1$ are homologous up to multiplicity, $\Sigma_1$ decomposes $\Gamma_1$ to a union of product sutured manifolds. Because $\Gamma(u)  \stackrel{\Sigma_2}{\overline \leadsto} \Gamma(v)$ is a nontrivial sutured decomposition, $\Sigma_1$ does not intersect a component of $s(\Gamma_0)$. Hence $\Sigma_1$ either avoids the $T\times I$ or meets the boundary parallel to the sutures of $T\times I$. If $\Sigma_1$ meets the boundary parallel to the sutures of $T\times I$, $\Sigma_1$ decompose $T\times I$ to a 4-ST with sutures isotopic to a curve $C_i$ on $\partial_i M$.

 This means that for each boundary component $\partial_i \Mb$, no matter whether $y_i$ is a vertex of the Thurston sphere or in an open edge or face, there is an element $z_i$ in $H_2(\Mb,\partial \Mb)$ such that $z_i$ is in an open edge or face and $\Gamma(z_i)$ contains a sutured guts component as a $T\times I$ or 4-ST with sutures isotopic to a curve $C_i$ on $\partial_i M$.

Since the double of a $T\times I$ with 2 sutures on a boundary component is a $T\times I$ with 2 parallel sutures on each boundary component, by cutting it along a maximal collection of horizontal surfaces, we have two 4-STs. So the horizontally prime guts of the double of a $T \times I$ is a union of two 4-STs. As for a 4-ST or 2-ST, the horizontally prime guts of the double of it is a union of two copies of itself. Therefore if the doubled guts of $u$ in $H_2(M,\partial M)$ contains a 4-ST with sutures isotopic to a curve $C_i$ on $\partial_i M$, we know that the guts $\Gamma(u)$ contain a $T\times I$ or a 4-ST with sutures isotopic to a curve $C_i$ on $\partial_i M$.

Suppose $z_i$ is in an open edge $e$ meeting a vertex $v$. By Theorem \ref{thm:invariancedoubledguts}, we know that the doubled guts of elements in an open Thurston cone is invariant up to equivalence in $D(M)$. So the doubled guts of every element in $e$ contains a 4-ST and hence its sutured guts contains a $T\times I$ or 4-ST, denoted as $G$. Moreover $C_i$ is isotopic to each suture of the $T\times I$ or 4-ST.

By Lemma \ref{lem:sequence}, this means the sutured guts of $v$ also contains a $T\times I$ or 4-ST with sutures isotopic to $C_i$. Take a $kv -z_i$ so that $kv -z_i$ is in an open edge $e'$ meeting $v$ and $\Gamma(v)$ can be decomposed by a surface $\Sigma$ to have $\Gamma(kv -z_i)$. Because $z_i$ and $v$ do not intersect $C_i$ homologically, we know $kv -z_i$ also does not intersect $C_i$ homologically. Therefore $\Gamma(kv -z_i)$ has a component as a $T\times I$ or 4-ST with sutures isotopic to a curve $C_i$. By Theorem \ref{thm:invariancedoubledguts}, we know that the sutured guts of every element in $e'$ contains a $T\times I$ or 4-ST with sutures isotopic to $C_i$. 

Keep doing this, we know that there exists a 2-dimensional subspace $V_i$ passing through $z_i$ and $v$ such that for each nonzero element $u$ in $V_i$, $\Gamma(u)$ contains a $T\times I$ or a 4-ST with sutures isotopic to $C_i$.

If $z_i$ is in an open face, we can apply the same argument for any 2-dimensional subspace $V_i$ passing through $z_i$. We pick a $V_i$ for $z_i$.

Since $V_1$ and $V_2$ are 2-dimensional subspaces and $H_2(\Mb, \partial \Mb)$ is a 3-dimensional space, $V_1\cap V_2$ is at least 1-dimensional. Let $z$ be a nonzero element in $V_1\cap V_2$. Then $\Gamma(z)$ contains $G_i$ which is a $T\times I$ or 4-ST such that $C_i$ is isotopic to each suture of $G_i$ for $i=1,2$. Because $\Mb$ is acylindrical, $C_1$ is not isotopic to $C_2$ and hence $G_1$ is not the same as $G_2$. We can pick a vertex $v$ such that the open edge or face containing $z$ meets $v$. Then by Lemma \ref{lem:sequence}, $\Gamma(v)$ contains $G_1,G_2$ and is at least depth 2. Note that the sutures of $G_1$ and $G_2$ are not isotopic. Then by Theorem \ref{thm:2TIor4ST}, the volume of $\Mb$ is at least $1.5 V_8$.
\end{proof}

\bibliographystyle{amsalpha}
%\nocite{*}
\bibliography{minimal_volume}

% \appendix
% \chapter{More Monticello Candidates}

\end{document}